\documentclass{amsart}

\usepackage{a4wide}
\usepackage{amsmath,amssymb}
\usepackage{amsthm}
\usepackage{aliascnt}
\usepackage{tikz-cd}
\usepackage{multicol}
\usepackage{xcolor}
\usepackage[colorlinks=true,linkcolor=black,citecolor=blue,urlcolor=blue]{hyperref}
\usepackage{cleveref}
\usepackage{quiver}
\usepackage{orcidlink}
\renewcommand{\MR}[1]{}

\theoremstyle{definition}
\newtheorem{definition}{Definition}[section]
\newaliascnt{remark}{definition}
\newtheorem{remark}[remark]{Remark}
\aliascntresetthe{remark}

\theoremstyle{plain}
\newaliascnt{proposition}{definition}
\newtheorem{proposition}[proposition]{Proposition}
\aliascntresetthe{proposition}
\newaliascnt{lemma}{definition}
\newtheorem{lemma}[lemma]{Lemma}
\aliascntresetthe{lemma}
\newaliascnt{corollary}{definition}
\newtheorem{corollary}[corollary]{Corollary}
\aliascntresetthe{corollary}
\newaliascnt{example}{definition}
\newtheorem{example}[example]{Example}
\aliascntresetthe{example}
\newaliascnt{theorem}{definition}
\newtheorem{theorem}[theorem]{Theorem}
\aliascntresetthe{theorem}
\newaliascnt{problem}{definition}
\newtheorem{problem}[problem]{Problem}
\aliascntresetthe{problem}

\crefname{definition}{definition}{definitions}
\Crefname{definition}{Definition}{Definitions}
\crefname{remark}{remark}{remarks}
\Crefname{remark}{Remark}{Remarks}
\crefname{proposition}{proposition}{propositions}
\Crefname{proposition}{Proposition}{Propositions}
\crefname{lemma}{lemma}{lemmas}
\Crefname{lemma}{Lemma}{Lemmas}
\crefname{corollary}{corollary}{corollaries}
\Crefname{corollary}{Corollary}{Corollaries}
\crefname{example}{example}{examples}
\Crefname{example}{Example}{Examples}
\crefname{theorem}{theorem}{theorems}
\Crefname{theorem}{Theorem}{Theorems}
\crefname{problem}{problem}{problems}
\Crefname{problem}{Problem}{Problems}

\begin{document}

\title[Transposed Poisson structures in graded and algebroid geometry]
{Transposed Poisson structures in graded and algebroid geometry}
\author[G.~{La Rosa}]{Gianmarco La Rosa\,\orcidlink{0000-0003-1047-5993}}
\address{Dipartimento di Matematica e Informatica, Universit\`a degli Studi di Palermo, via Archirafi 34, 90123, Palermo, Italy}
\email{gianmarco.larosa@unipa.it}
\subjclass[2020]{17B63, 17B70, 53D17, 17B66, 16W25}
\keywords{Transposed Poisson algebra, Gerstenhaber algebra, Lie--Rinehart algebra, $F$-algebroid, $\frac12$-derivation, Jacobi bracket}
\date{\today}
\raggedbottom
\begin{abstract}
	We ask whether the transposed Leibniz rule
$2z\cdot[x,y]=[z\cdot x,y]+[x,z\cdot y]$ survives the two passages that turn
Poisson algebras into geometry, the shift of degree which produces Gerstenhaber
algebras and the passage from algebras to algebroids. We introduce graded
	transposed Poisson algebras with a bracket of arbitrary degree and show that the
	parity of the degree decides which shape the answer takes. In even degree, over
	a unital base, such brackets are in bijection with homogeneous derivations, and
	on the exterior algebra of a vector bundle in degree zero they are exactly the
	gauge derivations of the bundle. In odd degree the bracket is forced to be the
	product with a single element of that degree, so a non-negatively graded algebra
	carries no non-zero transposed Gerstenhaber bracket. In particular, for a
	Lie--Rinehart algebra $(A,L)$, the Schouten--Rinehart Gerstenhaber algebra
	$\Lambda_A L$ is transposed if and only if the anchor and the bracket of $L$
	both vanish. We then introduce transposed $F$-algebroids and classify those with
	an identity: if the rank is at least two the anchor vanishes and the structure
	is the datum of a derivation of the multiplication, while if the rank is one the
	bundle is trivialised by the identity and the structure is, globally, the Jacobi
	bracket of a single vector field. For the tangent $F$-algebroid of an $F$-manifold $M$ this forces $\dim M=1$. 
\end{abstract}
\maketitle
\enlargethispage{3pt}

{\footnotesize\tableofcontents}

\section{Introduction}

Let $A$ be a commutative associative algebra equipped with a Lie bracket. The
classical Poisson compatibility asks that the adjoint operators act by
derivations of the commutative product,
\[
	[x,yz]=[x,y]z+y[x,z].
\]
Exchanging the roles of the two products in this rule leads to the
\emph{transposed Leibniz rule}
\begin{equation}\label{eq:intro-transposed}
	2z\cdot[x,y]=[z\cdot x,y]+[x,z\cdot y],
\end{equation}
and to the notion of a transposed Poisson algebra, introduced by Bai et al. \cite{Bai2023}. The two compatibilities are genuinely different. Under
\eqref{eq:intro-transposed} the multiplication operators are not derivations of
the Lie bracket but half derivations, and an algebra satisfying both rules has
vanishing mixed products. Nevertheless transposed Poisson algebras retain much
of the algebraic behaviour of Poisson algebras. They are closed under tensor
products, they are Koszul self-dual as an operad, they arise from
Novikov--Poisson algebras, and they satisfy a rich family of identities
\cite{Bai2023}. They admit a bialgebra theory through anti-pre-Lie bialgebras
\cite{Liu2024}, and, by a theorem of Sartayev \cite{Sartayev2024}, they coincide with the
commutative Gelfand--Dorfman algebras and are in particular $F$-manifold
algebras \cite{DamasBeites2024}.

The structure theory has developed rapidly along one line. The multiplication
operators of a transposed Poisson structure are $\tfrac12$-derivations of the
underlying Lie algebra, in the sense of Filippov, so a Lie algebra without
non-trivial $\tfrac12$-derivations admits no transposed Poisson product other
than the zero one \cite{Ferreira2021}. This criterion, whose ancestry lies in
the study of derivations of Lie algebras \cite{Leger1963}, is the engine of the
classification results now available for solvable and perfect Lie algebras
\cite{Kaygorodov2024perfect}, for solvable Lie algebras with filiform nilradical
\cite{Abdurasulov2024} and for Virasoro-type algebras
\cite{Kaygorodov2025VirasoroType}. A survey is \cite{DamasBeites2024}.

This literature is almost entirely algebraic. Poisson algebras, by contrast,
have two standard passages into geometry, and neither has been examined for the
transposed rule.

The first passage is a shift of degree. A Poisson bracket of degree $-1$ on a
graded commutative algebra is a Gerstenhaber bracket \cite{Gerstenhaber1963},
and this is where the Poisson world becomes geometric. The Hochschild cohomology of a regular affine algebra over a perfect field is
such an object \cite{Hochschild1962}, and so is the
exterior algebra $\Lambda_A L$ of any Lie--Rinehart algebra $(A,L)$, with the
Schouten--Rinehart bracket \cite{Rinehart1963,Huebschmann1998}. Multivector
fields with the Schouten--Nijenhuis bracket \cite{Marle1997,Vaisman1994} and the
exterior algebra of a Lie algebroid \cite{Kosmann-Schwarzbach1995} are the two
guiding examples, and the construction which produces a Gerstenhaber algebra out
of a Poisson algebra is classical \cite{Kosmann-Schwarzbach1996}. The transposed
counterpart of this passage has been opened only in the super case, by Abramov
and Liivapuu \cite{Abramov2024}, who introduced transposed Poisson
superalgebras and produced structures on the algebra of differential forms of a
manifold using the Lie derivative along a vector field. Those examples live in
degree zero, and no classification accompanies them.

The second passage is from algebras to algebroids. Poisson manifolds, Jacobi
manifolds and local Lie algebras \cite{Kirillov1976,Marle1991,Vaisman1994} are
all encoded by brackets on sections of a bundle, subject to a Leibniz rule
governed by an anchor. Since transposed Poisson algebras are $F$-manifold
algebras, and since $F$-algebroids are the algebroid-level version of
$F$-manifold algebras \cite{CruzMorales2019,CruzMorales2023}, there is a square
of inclusions whose fourth corner has no name.

\[
	\begin{tikzcd}
		{\text{transposed Poisson algebra}} &&& {\text{F-manifold algebras}} \\
		\\
		{\textbf{?}} &&& {\text{F-algebroids}}
		\arrow[hook, from=1-1, to=1-4]
		\arrow[dotted, tail reversed, from=1-4, to=3-4]
		\arrow[
    hook,
    from=3-1,
    to=3-4,
    shorten <=55pt,
    shorten >=15pt
]
	\end{tikzcd}
\]

The vertical arrow on the right is the passage from an $F$-manifold algebra to
an $F$-algebroid, an $F$-manifold algebra being an $F$-algebroid over a point.
The question mark is the structure which should occupy the corresponding
position for transposed Poisson algebras. The paper is organised around three
questions.

\begin{problem}\label{pb:graded}
Poisson algebras have graded counterparts in every degree, and the passage from
degree $0$ to degree $-1$ turns Poisson algebras into Gerstenhaber algebras.
Does the transposed Leibniz rule admit the same graded counterparts, and what
are they?
\end{problem}

\begin{problem}\label{pb:LR}
When is the canonical Gerstenhaber algebra $(\Lambda_A L,\wedge,[-,-]_{SR})$
associated with the Schouten--Rinehart bracket transposed Gerstenhaber, and
which transposed Gerstenhaber brackets can be placed on the graded commutative
algebra $\Lambda_A L$?
\end{problem}

\begin{problem}\label{pb:algebroid}
Which structure occupies the empty box of the diagram, and what do its objects
look like over a base of positive dimension?
\end{problem}

\Cref{pb:graded} is answered in \Cref{sec:graded}. We define graded transposed
Poisson algebras of arbitrary degree $k$, so that $k=0$ recovers the super case
of \cite{Abramov2024} and $k=-1$ gives a transposed analogue of Gerstenhaber
algebras. The sign appearing in the definition is not a convention, since it is
the only one compatible with graded skew-symmetry, and the parity of $k$ turns
out to decide the shape of the classification, the classifying object itself
depending on $k$. For $k$ even the theory is faithful to the ungraded one, and
over a unital base transposed brackets are in bijection with homogeneous
derivations of degree $k$ (see \Cref{thm:unital-even}). For $k$ odd, the Wronskian
construction breaks down, and the bracket is forced to be
\[
	[x,y]=(-1)^{|y|}(xy)\theta,
	\qquad \theta=[1,1]\in A^{k},
\]
(see \Cref{thm:unital-odd}). A non-negatively graded algebra therefore carries no
non-zero transposed bracket of negative odd degree (see \Cref{cor:rigidity}). The
dividing line is quantified by two identities valid in every degree, whose
correction terms carry the complementary factors $\tfrac12(1-(-1)^k)$ and
$1+(-1)^k$ (see \Cref{prop:master-identity} and \Cref{lem:parity-split}).

\Cref{pb:LR} is answered in \Cref{sec:gerstenhaber}, in the reverse order to the
one in which it is stated. The second half is
\Cref{thm:transposed-gerstenhaber-unital}, which describes all transposed
Gerstenhaber brackets on a unital graded commutative algebra and leaves only the
zero bracket on $\Lambda_A L$. The first half follows. Writing $\Theta$ for the
failure of a Gerstenhaber algebra to satisfy \eqref{eq:intro-transposed}, a
short computation gives
\[
	\Theta(x,y,z)=(-1)^{|x|(|y|-1)}[z,y]\,x+(-1)^{|z|(|x|-1)}[x,z]\,y
\]
(see \Cref{lem:theta-formula}), and evaluating at $x=1$ shows that a unital
Gerstenhaber algebra is transposed precisely when its bracket vanishes (see
\Cref{thm:unital-gerstenhaber-rigid}). Hence, for a Lie--Rinehart algebra
$(A,L)$, the algebra $\Lambda_A L$ is transposed Gerstenhaber if and only if the
anchor and the bracket of $L$ both vanish (see \Cref{thm:LR-complete}). A weaker
obstruction, whose proof stays in degree one and never evaluates at the unit, is
recorded in \Cref{prop:exterior-obstruction}.

The positive counterpart lives in degree zero, and is the content of
\Cref{sec:geometry}. On the algebra of functions of a manifold, transposed
Poisson structures correspond to vector fields, and they are exactly the Jacobi
brackets with vanishing bivector part (see \Cref{prop:transposed-manifold}). On
the exterior algebra of a vector bundle $E$ the same question in degree zero has
the answer $\Gamma(\mathfrak D(E))$, the gauge derivations of $E$ (see
\Cref{thm:degree-zero-bundle}). This classifies the examples of
\cite{Abramov2024} as the case $E=T^*M$ with $\mathfrak D$-section the Lie
derivative, and it shows that in all other degrees nothing survives.

\Cref{pb:algebroid} is answered in \Cref{sec:Falgebroids}. The empty box is
filled by transposed $F$-algebroids, \Cref{def:transposed-F}, which contain
transposed Poisson algebras as the objects over a point and for which the
Hertling--Manin relation need not be imposed (see \Cref{prop:HM-free}). Over a
base of positive dimension they are scarce, and the reason is a single identity,
obtained by inserting $fZ$ into \eqref{eq:intro-transposed},
\[
	a_A(X)(f)\,(Z\cdot_A Y)=a_A(Y)(f)\,(Z\cdot_A X)
\]
(see \Cref{lem:anchor-obstruction}). It forces the multiplication to have rank one
wherever the anchor survives, and then the whole structure follows. A transposed
$F$-algebroid with identity and rank at least two has vanishing anchor and is
the datum of a derivation of the multiplication; one of rank one is trivialised
by the identity and carries the bracket
\[
	[fe,ge]=\bigl(fV(g)-gV(f)\bigr)e
\]
of a single vector field $V$ (see \Cref{thm:rank-one}, \Cref{thm:local-model}
and \Cref{cor:unital-dichotomy}). For the tangent $F$-algebroid of an
$F$-manifold this means $\dim M=1$ (see \Cref{cor:F-manifold}), with no
semisimplicity assumption. The rank one model is again a Jacobi bracket with
vanishing bivector part, so the algebroid level enlarges nothing (see
\Cref{prop:jacobi}).

The three answers are instances of a single phenomenon. Whenever a unit is
present, \eqref{eq:intro-transposed} evaluated at $1$ collapses the bracket onto
the single operator $D=[1,-]$, and the entire structure is the structure of $D$.
On functions $D$ is a vector field. On the exterior algebra of a vector bundle
in degree zero it is a gauge derivation. On an $F$-algebroid it is a gauge
derivation whose symbol is the anchor. The rigidity statements of this paper are
the record of what such a $D$ cannot be, and the classification statements are
the record of what it can. In odd degree $D$ is still there, but it is forced to
be the multiplication by the single odd element $\theta=[1,1]$ rather than a
derivation, so the bracket cannot be of Wronskian type; and since $\theta$ has
to live in degree $k$, nothing survives on a non-negatively graded algebra,
which is why the Gerstenhaber world is empty. On a bundle of rank at least two
there is no room for a symbol, which is why the anchor dies.

\Cref{sec:preliminaries} fixes conventions and recalls transposed Poisson
algebras, graded Poisson and Gerstenhaber brackets, and Lie--Rinehart algebras
with the Schouten--Rinehart bracket. The sections which follow are the ones
announced above, and \Cref{sec:problems} collects the questions left open.
Throughout, the base field has characteristic zero and commutative means
commutative and associative. Every classification below rests on the evaluation
of \eqref{eq:intro-transposed} at a unit, so the results are stated for unital
algebras. What happens without a unit is \Cref{pb:nonunital}.

Throughout, the base field has characteristic zero and commutative means
commutative and associative. Every classification below rests on the evaluation
of \eqref{eq:intro-transposed} at a unit, so the results are stated for unital
algebras; what happens without a unit is \Cref{pb:nonunital}.

\section{Preliminaries}\label{sec:preliminaries}

Throughout the paper, unless otherwise stated, all vector spaces are taken over
a field $\mathbb K$ of characteristic zero. Characteristic zero is used in an
essential way: the transposed Leibniz rule carries a factor $2$, and the
arguments below divide by it repeatedly. By a commutative algebra we shall
always mean a commutative associative algebra, not necessarily unital.

We begin by recalling the notion introduced in \cite{Bai2023}. Let $A$ be a
vector space endowed with two bilinear operations $\cdot$ and $[-,-]$ in $A$.
The triple $(A,\cdot,[-,-])$ is called a \emph{transposed Poisson algebra} if
$(A,\cdot)$ is a commutative algebra, $(A,[-,-])$ is a Lie algebra, and the
following compatibility relation holds:
\begin{equation}\label{eq:transposed-leibniz}
	2z\cdot [x,y]=[z\cdot x,y]+[x,z\cdot y],
	\qquad x,y,z\in A .
\end{equation}
This identity expressed by \Cref{eq:transposed-leibniz} will be referred to as
the \emph{transposed Leibniz rule}. The terminology is justified by comparison
with the usual Poisson identity
\[
	[x,yz]=[x,y]z+y[x,z].
\]
In a Poisson algebra, the adjoint action of the Lie bracket acts by derivations
of the commutative product. In \Cref{eq:transposed-leibniz}, instead,
multiplication by $z$ acts on the Lie bracket as one half of a derivation:
\[
	[M_zx,y]+[x,M_zy]=2M_z[x,y],
\]
where $M_z(x)=zx$. Thus the multiplication operators are not derivations of the
Lie algebra, but $\frac12$-derivations in the sense suggested by the displayed
identity. The basic source of examples is parallel to the classical construction
of Poisson brackets from derivations. Let $(A,\cdot)$ be a commutative algebra
and let $D\in\operatorname{Der}(A)$. Define
\begin{equation}\label{eq:derivation-bracket}
	[x,y]_D=xD(y)-yD(x),
	\qquad x,y\in A .
\end{equation}
Then $(A,\cdot,[-,-]_D)$ is a transposed Poisson algebra. Indeed, the Jacobi
identity for \Cref{eq:derivation-bracket} is the standard one for brackets of
this type, while the compatibility follows immediately:
\[
	\begin{aligned}
		\relax[zx,y]_D+[x,zy]_D
		  & = zxD(y)-yD(zx)+xD(zy)-zyD(x)               \\
		  & = zxD(y)-zyD(x)-xyD(z)+xzD(y)+xyD(z)-zyD(x) \\
		  & =2z\bigl(xD(y)-yD(x)\bigr)                  \\
		  & =2z[x,y]_D .                                
	\end{aligned}
\]
Thus a single derivation produces the transposed analogue of the Poisson bracket
constructed from two commuting derivations. When the commutative algebra is
unital, this construction is exhaustive. More precisely, let $A$ be a unital
commutative algebra. The assignment
\[
	D\longmapsto [-,-]_D,\qquad
	[x,y]_D=xD(y)-yD(x),
\]
gives a bijection between $\operatorname{Der}(A)$ and the set of Lie brackets on
$A$ which make $(A,\cdot,[-,-])$ into a transposed Poisson algebra. The inverse
map sends a transposed Poisson bracket to the derivation $D(x)=[1,x]$. Indeed,
applying \Cref{eq:transposed-leibniz} with the unit gives
\[
	[x,y]=x[1,y]-y[1,x],
\]
and the same identity also implies
\[
	[1,xy]=x[1,y]+y[1,x].
\]
Hence $D=[1,-]$ is a derivation and the bracket is recovered from $D$ by
\Cref{eq:derivation-bracket}. This is the statement of Proposition~2.4 in \cite{Bai2023}.

We record a few identities which will be used later. If $(A,\cdot,[-,-])$ is a
transposed Poisson algebra, then, for all $x,y,z,h,u,v\in A$,
\begin{align}
	x[y,z]+y[z,x]+z[x,y]             & =0, \label{eq:cyclic-basic}         \\
	[h[x,y],z]+[h[y,z],x]+[h[z,x],y] & =0, \label{eq:cyclic-h}             \\
	[hx,[y,z]]+[hy,[z,x]]+[hz,[x,y]] & =0, \label{eq:hom-lie-id}           \\
	[h,x][y,z]+[h,y][z,x]+[h,z][x,y] & =0, \label{eq:farkas-id}            \\
	[xu,yv]+[xv,yu]                  & =2uv[x,y]. \label{eq:mixed-product} 
\end{align}
The first of these is obtained by writing \Cref{eq:transposed-leibniz} for the
three cyclic permutations of $(x,y,z)$ and summing: the six brackets on the
right-hand side cancel in pairs by skew-symmetry. The remaining identities are
consequences of the transposed Leibniz rule and the Jacobi identity, see
\cite[Thm.~2.7]{Bai2023}; they express the rather rigid interaction between the
two operations.

Two simple consequences are worth isolating. First, for every $h\in A$, the
rescaled bracket
\begin{equation}\label{eq:rescaling}
    [x,y]_h=h[x,y]
\end{equation}
again defines a transposed Poisson algebra on the same commutative algebra $A$,
the Jacobi identity being exactly \Cref{eq:cyclic-h}. Second, the linear map
\[
	\varphi_h:A\to A,\qquad \varphi_h(x)=hx,
\]
turns $(A,[-,-])$ into a Hom-Lie algebra, since \Cref{eq:hom-lie-id} is
precisely the Hom-Jacobi identity for $\varphi_h$.

Transposed Poisson algebras are also stable under tensor products. If
$(A_1,\cdot_1,[-,-]_1)$ and $(A_2,\cdot_2,[-,-]_2)$ are transposed Poisson
algebras, then $A_1\otimes A_2$ carries a transposed Poisson structure with
product
\[
	(x_1\otimes x_2)(y_1\otimes y_2)=x_1y_1\otimes x_2y_2
\]
and bracket
\[
	[x_1\otimes x_2,y_1\otimes y_2]
	=
	[x_1,y_1]_1\otimes x_2y_2
	+
	x_1y_1\otimes [x_2,y_2]_2 .
\]
This is formally identical to the usual tensor product bracket for Poisson
algebras, but the verification uses the identities above rather than the
ordinary Leibniz rule.

Finally, the transposed and the ordinary Poisson compatibilities are essentially
orthogonal. Suppose that $(A,\cdot,[-,-])$ is at the same time a Poisson algebra
and a transposed Poisson algebra. Then all mixed products vanish,
\[
	x[y,z]=[xy,z]=0,
	\qquad x,y,z\in A ,
\]
and the converse implication is immediate. This is \cite[Proposition.~2.13]{Bai2023}.
The proof inserts the ordinary Leibniz rule into \Cref{eq:transposed-leibniz} and
then appeals to the cyclic identity \Cref{eq:cyclic-basic}. The intersection of
the two theories is therefore degenerate in a precise sense: outside the
vanishing of all mixed products, the transposed Leibniz rule is not a disguised
form of the usual Poisson rule.

The next section introduces the graded version of \Cref{eq:transposed-leibniz}.
The preceding identities will serve as the ungraded model; the main point will
be to insert the signs forced by a bracket of arbitrary degree.

Now we are going to recall Poisson and Gerstenhaber brackets. To do this, we shall use cohomological grading. If $A=\bigoplus_{i\in\mathbb Z}A^i$ is a
graded vector space and $x\in A^{|x|}$ is homogeneous, its degree will be
denoted by $|x|$. All identities involving signs are understood for homogeneous
elements and are then extended by linearity.

A graded algebra $(A,\cdot)$ is called \emph{graded commutative} if
$xy=(-1)^{|x||y|}yx$. A bilinear bracket $[-,-]$ on $A$ is said to have
\emph{degree} $k$ if $[A^p,A^q]\subseteq A^{p+q+k}$. It is a \emph{graded Lie
bracket} of degree $k$ if
\[
	[x,y]=-(-1)^{(|x|+k)(|y|+k)}[y,x]
\]
and
\[
	[x,[y,z]]
	=
	[[x,y],z]
	+(-1)^{(|x|+k)(|y|+k)}[y,[x,z]] .
\]
Equivalently, the shifted space $A[k]$ is a graded Lie algebra for the induced
bracket of degree zero.

A \emph{graded Poisson algebra of degree $k$} is a graded commutative
associative algebra $A$, together with a graded Lie bracket of degree $k$, such
that the adjoint operators act as graded derivations of the product. Thus, for
homogeneous $x,y,z\in A$,
\begin{equation}\label{eq:graded-poisson-leibniz}
	[x,yz]=[x,y]z+(-1)^{(|x|+k)|y|}y[x,z].
\end{equation}
For $k=0$ this is the usual notion of a graded Poisson algebra. For $k=-1$ one
obtains the convention traditionally used for Gerstenhaber algebras.

A \emph{Gerstenhaber algebra} is therefore a graded commutative algebra
$\mathcal G$ endowed with a bracket $[-,-]$ of degree $-1$ such that
$\mathcal G[1]$ is a graded Lie algebra and
\[
	[x,yz]=[x,y]z+(-1)^{(|x|-1)|y|}y[x,z].
\]
Equivalently, for homogeneous elements,
\[
	[x,y]=-(-1)^{(|x|-1)(|y|-1)}[y,x],
\]
and the bracket satisfies the Jacobi identity with the same shifted signs. This
is the convention used throughout the sequel. We shall repeatedly use the
resulting Leibniz rule in the first argument,
\begin{equation}\label{eq:gerst-leibniz-first}
	[yz,x]=(-1)^{|z|(|x|-1)}[y,x]z+y[z,x],
\end{equation}
which follows from \Cref{eq:graded-poisson-leibniz} and skew-symmetry.

The original example is the Hochschild cohomology of an associative algebra,
with the cup product and the Gerstenhaber bracket \cite{Gerstenhaber1963}. In
geometric situations, the guiding example is the algebra of multivector fields,
with wedge product and Schouten bracket. More generally, if $(A,L)$ is a
Lie--Rinehart algebra (whose definition will be remembered shortly), then the exterior $A$-algebra $\Lambda_A L$, with $L$
placed in degree one, carries a canonical Gerstenhaber bracket extending the Lie
bracket of $L$ and the action of $L$ on $A$ \cite{Huebschmann1998}.

We shall later compare \Cref{eq:graded-poisson-leibniz} with its transposed
analogue. The degree $k=-1$ will be the case of primary interest: it is
precisely there that the transposed rule interacts with Gerstenhaber algebras
and, in particular, with the Gerstenhaber algebra $\Lambda_A L$ attached to a
Lie--Rinehart pair.

In order to recall the definition of Lie--Rinehart algebras, let $R$ be a commutative ring and let $A$ be a commutative $R$-algebra. A
\emph{Lie--Rinehart algebra} over $(R,A)$ is an $R$-Lie algebra $L$, endowed
with an $A$-module structure and an $R$-Lie algebra morphism
\[
	\rho:L\longrightarrow \operatorname{Der}_R(A),
\]
called the \textit{anchor}, such that
\[
	[X,aY]=a[X,Y]+\rho(X)(a)Y,
	\qquad
	\rho(aX)(b)=a\rho(X)(b),
\]
for all $a,b\in A$ and $X,Y\in L$. We shall often write $X(a)$ instead of
$\rho(X)(a)$.

The notion has been introduced several times and under several names, from
Herz's pseudo-algebras of Lie and Palais's Lie rings with operators to
Rinehart's $(R,A)$-Lie algebras \cite{Herz1953_1,Herz1953_2,Palais1961,Rinehart1963}. Its
homological theory is developed in
\cite{Huebschmann1990,Huebschmann1998,Huebschmann1999}, and its smooth
counterpart, the theory of Lie algebroids, in
\cite{Mackenzie2005,CrainicFernandes2011}.

The basic example is obtained from a smooth manifold $M$, by taking
$R=\mathbb R$, $A=C^\infty(M)$ and $L=\mathfrak X(M)$ with the usual bracket of
vector fields and the natural action on functions. In
this case the construction below recovers the Schouten bracket on multivector
fields.

\begin{remark}\label{rmk:unital-convention}
	Following \cite{Rinehart1963,Huebschmann1998}, the base $A$ of a Lie--Rinehart
	algebra is unital and $L$ is a unitary $A$-module. In particular
	$\Lambda^0_AL=A$, the algebra $\Lambda_AL$ is generated in degrees $0$ and $1$,
	and the universal property which extends the bracket of $L$ and the anchor to
	$\Lambda_AL$ is available. We keep this convention throughout, and we return to
	the non-unital situation in \Cref{pb:nonunital}.
\end{remark}

Let $(A,L)$ be a Lie--Rinehart algebra. We place $L$ in degree $1$ and consider
the exterior $A$-algebra
\[
	\Lambda_A L=\bigoplus_{p\geq 0}\Lambda_A^p L.
\]
It is a graded commutative algebra with respect to the wedge product. The bracket of $L$ extends uniquely to a bracket $[-,-]_{SR}$ of degree $-1$ on
$\Lambda_AL$ which is graded antisymmetric, restricts on $\Lambda_A^0L=A$ to the
action of $L$ on $A$, and satisfies the Leibniz rule
\[
	[U,V\wedge W]_{SR}
	=
	[U,V]_{SR}\wedge W
	+(-1)^{(|U|-1)|V|}V\wedge[U,W]_{SR}
\]
for homogeneous $U,V,W\in\Lambda_AL$. Together with the wedge product it makes
$\Lambda_AL$ into a Gerstenhaber algebra, so that $\Lambda_AL[1]$ is a graded Lie
algebra \cite{Huebschmann1998}. On decomposable elements of positive degree the
extension reads
\[
	\begin{aligned}
		  & [X_1\wedge\cdots\wedge X_p, 
		Y_1\wedge\cdots\wedge Y_q]_{SR}     \\
		  & \qquad =                    
		\sum_{i,j}
		(-1)^{i+j}
		[X_i,Y_j]\wedge
		X_1\wedge\cdots\widehat X_i\cdots\wedge X_p
		\wedge
		Y_1\wedge\cdots\widehat Y_j\cdots\wedge Y_q .
	\end{aligned}
\]

We shall use this Gerstenhaber algebra as the main test case for the transposed
condition in degree $-1$. The  $[X,a]_{SR}=X(a)$ and $[a,b]_{SR}=0$
will be especially relevant: they make the anchor visible inside the bracket and
therefore allow the transposed Leibniz rule to impose direct restrictions on the
Lie--Rinehart structure.

\section{Graded transposed Poisson algebras}\label{sec:graded}

We now pass from the ungraded setting to the graded one. A super version of
transposed Poisson algebras was introduced by Abramov and Liivapuu
\cite{Abramov2024}; in their setting the grading is by $\mathbb Z_2$ and the
Lie bracket is even. Here we work with $\mathbb Z$-graded algebras and allow
the bracket to have an arbitrary degree.

Thus, throughout this section, graded means $\mathbb Z$-graded. The super case
is obtained by reducing degrees modulo two when the bracket has even degree.
The case of main interest for us is $k=-1$, where the graded Poisson identity
is the Gerstenhaber Leibniz rule.

\begin{definition}\label{def:graded-transposed}
	Let $A=\bigoplus_{i\in\mathbb Z}A^i$ be a graded vector space. A
	\emph{graded transposed Poisson algebra of degree $k$} is a triple
	$(A,\cdot,[-,-])$ such that $(A,\cdot)$ is a graded commutative associative
	algebra, $[-,-]$ is a graded Lie bracket of degree $k$, and
	\begin{equation}\label{eq:graded-transposed-leibniz}
		2z[x,y]=[zx,y]+(-1)^{|z|(|x|+k)}[x,zy]
	\end{equation}
	for all homogeneous $x,y,z\in A$.
\end{definition}

The sign in the second term is forced by the degree of the adjoint operator:
for a bracket of degree $k$, the map $[x,-]$ has degree $|x|+k$. The following
proposition makes this precise: no other choice of sign is compatible with the graded
skew-symmetry of the bracket.

\begin{proposition}\label{lem:sign-forced}
	For homogeneous $x,y,z\in A$, \Cref{eq:graded-transposed-leibniz} written for
	the triple $(x,y,z)$ and \Cref{eq:graded-transposed-leibniz} written for the
	triple $(y,x,z)$ are equivalent. If the sign $(-1)^{|z|(|x|+k)}$ is replaced by
	$(-1)^{|z||x|}$, the two resulting identities differ by the factor
	$(-1)^{k|z|}$, hence they are equivalent for all $x,y,z$ only when $k$ is even.
\end{proposition}

\begin{proof}
	Write \Cref{eq:graded-transposed-leibniz} for $(y,x,z)$ and rewrite its
	right-hand side by graded skew-symmetry:
    \begin{align*}
      [zy,x]&=-(-1)^{(|z|+|y|+k)(|x|+k)}[x,zy],\\
      (-1)^{|z|(|y|+k)}[y,zx]
		&=-(-1)^{|z|(|y|+k)+(|y|+k)(|z|+|x|+k)}[zx,y].
    \end{align*}

	The two exponents reduce modulo two to
	$|z|(|x|+k)+(|x|+k)(|y|+k)$ and to $(|x|+k)(|y|+k)$ respectively. Since
	$2z[y,x]=-(-1)^{(|x|+k)(|y|+k)}2z[x,y]$, the identity for $(y,x,z)$ is
	$-(-1)^{(|x|+k)(|y|+k)}$ times the identity for $(x,y,z)$. The second assertion
	follows from the same computation, in which the two reduced exponents acquire
	the extra summand $k|z|$.
\end{proof}

\begin{remark}
	For $k=0$, reducing degrees modulo two gives the compatibility condition of
	transposed Poisson superalgebras in \cite{Abramov2024}. Thus the super theory
	with even bracket is contained in the present definition. For $k=-1$ one obtains
	the transposed analogue of a Gerstenhaber algebra.
\end{remark}

Two different classification problems are attached to
\Cref{def:graded-transposed}, according to which of the two operations is held
fixed. Given a graded commutative associative algebra $(A,\cdot)$, a
\emph{transposed bracket of degree $k$} on it is a graded Lie bracket of degree
$k$ satisfying \eqref{eq:graded-transposed-leibniz}. Given a graded Lie algebra
$(A,[-,-])$ of degree $k$, a \emph{transposed product} on it is a graded
commutative associative product satisfying
\eqref{eq:graded-transposed-leibniz}. The zero bracket is a transposed bracket
on every $(A,\cdot)$, and the zero product is a transposed product on every
$(A,[-,-])$, so both problems have a solution and the question is always whether
there are others. We avoid calling either of these solutions trivial, since the
two vanishing conditions are independent: the zero product satisfies
\eqref{eq:graded-transposed-leibniz} whatever the bracket is, and imposes
nothing on it.

Now we are ready to look at some first consequences of these assumptions. Let $(A,\cdot,[-,-])$ be a graded transposed Poisson algebra of degree $k$.
For a homogeneous element $z\in A$, we denote by
\[
	M_z:A\longrightarrow A,\qquad M_z(x)=zx,
\]
the multiplication operator. Thus $M_z$ is homogeneous of degree $|z|$.

\begin{definition}
	Let $(A,[-,-])$ be a graded Lie algebra with bracket of degree $k$. A
	homogeneous linear map $\varphi:A\to A$ of degree $r$ is called a
	\emph{graded $\frac12$-derivation} if
	\[
		\varphi([x,y])
		=
		\frac12\Big(
		[\varphi(x),y]
		+
		(-1)^{r(|x|+k)}[x,\varphi(y)]
		\Big)
	\]
	for all homogeneous $x,y\in A$. It is called \emph{trivial} if it is a scalar
	multiple of the identity.
\end{definition}

Trivial graded $\frac12$-derivations are homogeneous of degree zero, and they
are always present. The relevant hypothesis is therefore the absence of
$\frac12$-derivations other than these. Written in terms of the multiplication
operators, \eqref{eq:graded-transposed-leibniz} says precisely that $M_z$ is a
graded $\frac12$-derivation of degree $|z|$ for every homogeneous $z\in A$, and
this is what makes the hypothesis effective.

\begin{proposition}\label{prop:no-half-derivations}
Let $(A,[-,-])$ be a graded Lie algebra of degree $k$ with $\dim A>1$. If every
homogeneous graded $\frac12$-derivation of $A$ is trivial, then the only
transposed product on $(A,[-,-])$ is the zero product.
\end{proposition}

\begin{proof}
Let $\cdot$ be such a product. As just observed, $M_z$ is a homogeneous
$\frac12$-derivation of degree $|z|$ for every homogeneous $z$. A scalar
multiple of the identity is homogeneous of degree zero, so $M_z=0$ whenever
$|z|\neq0$. By graded commutativity, $zx=0$ as soon as one of the two factors
has nonzero degree.

There remains the component $A^0$. For $z,x\in A^0$ write
$M_z=\kappa_z\,\mathrm{id}$ and $M_x=\kappa_x\,\mathrm{id}$, so that
$zx=\kappa_z x=\kappa_x z$. If $\dim A^0\geq2$, choosing $z$ and $x$ linearly
independent gives $\kappa_z=\kappa_x=0$. If $\dim A^0=1$, say
$A^0=\langle e\rangle$, then $\dim A>1$ provides a nonzero homogeneous $x$ with
$|x|\neq0$, and $0=ex=\kappa_e x$ forces $\kappa_e=0$. In either case the
product vanishes identically.
\end{proof}

\begin{remark}
	In the ungraded theory, the preceding observation is one of the basic tools for
	studying transposed Poisson structures on a fixed Lie algebra: multiplication
	operators are $\frac12$-derivations, and the absence of non-trivial
	$\frac12$-derivations forces the transposed product to vanish
	\cite{Ferreira2021}. The same mechanism remains valid in the graded setting,
	with the sign dictated by the degree of the bracket.
\end{remark}

The cyclic identity \Cref{eq:cyclic-basic} is the first place where the parity
of $k$ becomes visible. The following identity holds in all degrees and isolates
the correction term.

\begin{proposition}\label{prop:master-identity}
	Let $(A,\cdot,[-,-])$ be a graded transposed Poisson algebra of degree $k$.
	Then, for all homogeneous $x,y,z\in A$,
	\[
		x[y,z]
		-(-1)^{|x||y|}y[x,z]
		+(-1)^{|x||z|+(|y|+k)(|z|+k)}z[x,y]
		=
		\frac{1-(-1)^{k}}{2}\,(-1)^{E}\,[zy,x],
	\]
	where $E=|x||y|+(|x|+k)(|z|+k)+|y||z|$.
\end{proposition}

\begin{proof}
	Applying \eqref{eq:graded-transposed-leibniz} to the triple obtained from
	$(x,y,z)$ by the substitution $z\mapsto x$, $x\mapsto y$, $y\mapsto z$ gives
	$2x[y,z]=[xy,z]+(-1)^{|x|(|y|+k)}[y,xz]$, and rewriting the last bracket by
	graded skew-symmetry, in which the second argument has degree $|x|+|z|$, the
	two occurrences of $|x|$ cancel and we obtain
	\begin{equation}\label{eq:first-slot}
		[xy,z]=2x[y,z]+(-1)^{(|y|+k)(|z|+k)}[xz,y].
	\end{equation}
	Applying \Cref{eq:first-slot} to the triple $(z,x,y)$ and using
	$xz=(-1)^{|x||z|}zx$, we may substitute the result into \Cref{eq:first-slot} and
	obtain
	\[
		[xy,z]
		=2x[y,z]
		+(-1)^{(|y|+k)(|z|+k)+|x||z|}
		\Bigl(2z[x,y]+(-1)^{(|x|+k)(|y|+k)}[zy,x]\Bigr).
	\]
	On the other hand $[xy,z]=(-1)^{|x||y|}[yx,z]$, and \Cref{eq:first-slot} applied
	to $(y,x,z)$, together with $[yz,x]=(-1)^{|y||z|}[zy,x]$, gives
	\[
		[xy,z]
		=(-1)^{|x||y|}
		\Bigl(2y[x,z]+(-1)^{(|x|+k)(|z|+k)+|y||z|}[zy,x]\Bigr).
	\]
	Comparing the two expressions, the coefficients of $[zy,x]$ are $(-1)^{E'}$ and
	$(-1)^{E}$ with
	\[
		E'=(|y|+k)(|z|+k)+|x||z|+(|x|+k)(|y|+k),
		\qquad
		E=|x||y|+(|x|+k)(|z|+k)+|y||z| .
	\]
	Expanding and using $k^2\equiv k$ modulo two gives $E'\equiv E+k$. Hence
	\[
		2x[y,z]-2(-1)^{|x||y|}y[x,z]
		+2(-1)^{|x||z|+(|y|+k)(|z|+k)}z[x,y]
		=\bigl((-1)^{E}-(-1)^{E'}\bigr)[zy,x],
	\]
	and $(-1)^{E}-(-1)^{E'}=(1-(-1)^{k})(-1)^{E}$. Dividing by $2$ gives the claim.
\end{proof}

\begin{corollary}\label{cor:cyclic-even}
	If $k$ is even, then for all homogeneous $x,y,z\in A$,
	\[
		(-1)^{|x||z|}x[y,z]
		+
		(-1)^{|y||x|}y[z,x]
		+
		(-1)^{|z||y|}z[x,y]
		=0 .
	\]
\end{corollary}

\begin{proof}
	For $k$ even the right-hand side of \Cref{prop:master-identity} vanishes and
	$(-1)^{(|y|+k)(|z|+k)}=(-1)^{|y||z|}$. Multiplying the resulting identity by
	$(-1)^{|x||z|}$ and replacing $-(-1)^{|x||y|}y[x,z]$ by
	$(-1)^{|x||y|+|x||z|}y[z,x]$ gives the displayed formula.
\end{proof}

Thus the cyclic identity of the even super case, and with it the identities
\Cref{eq:cyclic-h,eq:hom-lie-id,eq:farkas-id,eq:mixed-product} which depend on
it, is tied to the fact that the Lie bracket is even. For odd $k$, and in
particular for the Gerstenhaber degree $k=-1$, the correction term
$[zy,x]$ survives, and no such cyclic identity holds in general.

\subsection{The unital case}

We now turn to the unital case, which admits a complete description in every
degree. Both parities are governed by \Cref{lem:parity-split} below, whose
right-hand side carries the factor $1+(-1)^{k}$: for $k$ even the terms in
$D(zy)$ double and force $D=[1,-]$ to be a derivation, while for $k$ odd they
cancel and leave a relation from which $D$ is recovered from the single element
$[1,1]$. The two theorems below are the two branches of this alternative. The
complementary factor $1-(-1)^k$ appeared in \Cref{prop:master-identity}, where
it measured the failure of the cyclic identity.

\begin{lemma}\label{lem:unital-D}
	Let $(A,\cdot,[-,-])$ be a unital graded transposed Poisson algebra of degree
	$k$, and set $D(x)=[1,x]$. Then, for every homogeneous $z,y\in A$,
	\[
		[z,y]
		=
		2zD(y)-(-1)^{k|z|}D(zy).
	\]
	In particular, the bracket is determined by the operator $D=[1,-]$.
\end{lemma}

\begin{proof}
	Apply the transposed Leibniz rule with $x=1$. Since $1z=z$, one obtains
	\[
		2z[1,y]=[z,y]+(-1)^{|z|k}[1,zy],
	\]
	which is the desired formula.
\end{proof}

\begin{lemma}\label{lem:parity-split}
	Let $(A,\cdot,[-,-])$ be a unital graded transposed Poisson algebra of degree
	$k$ and let $D=[1,-]$. Then, for all homogeneous $z,y\in A$,
	\begin{equation}\label{eq:parity-split}
		2zD(y)+2(-1)^{(|z|+k)(|y|+k)}yD(z)
		=(-1)^{k|z|}\bigl(1+(-1)^{k}\bigr)D(zy).
	\end{equation}
\end{lemma}

\begin{proof}
	Write $a=|z|$ and $b=|y|$. By \Cref{lem:unital-D} applied to the pair $(y,z)$,
	and by $D(yz)=(-1)^{ab}D(zy)$,
	\[
		[y,z]=2yD(z)-(-1)^{kb+ab}D(zy).
	\]
	Substituting this into $[z,y]=-(-1)^{(a+k)(b+k)}[y,z]$ and comparing with
	\Cref{lem:unital-D} applied to $(z,y)$ gives
	\[
		2zD(y)+2(-1)^{(a+k)(b+k)}yD(z)
		=\Bigl((-1)^{ka}+(-1)^{(a+k)(b+k)+kb+ab}\Bigr)D(zy).
	\]
	Expanding $(a+k)(b+k)\equiv ab+ak+bk+k$ modulo two, the second exponent reduces
	to $ka+k$, and the two coefficients are $(-1)^{ka}$ and $(-1)^{ka+k}$.
\end{proof}

\begin{lemma}\label{lem:even-no-signs}
	Let $k$ be even and let $D$ be a homogeneous derivation of degree $k$ of a
	graded commutative associative algebra $(A,\cdot)$. Then $|D(x)|$ and $|x|$ have
	the same parity, so that
	\begin{equation}\label{eq:wronskian-nosign}
		x\,D(y)-(-1)^{(|x|+k)(|y|+k)}y\,D(x)=x\,D(y)-D(x)\,y,
		\qquad
		D(xy)=D(x)\,y+x\,D(y),
	\end{equation}
	both without signs, and the sign in \eqref{eq:graded-transposed-leibniz} is
	$(-1)^{|z||x|}$.
\end{lemma}

\begin{proof}
	Since $k$ is even, $|D(x)|=|x|+k\equiv|x|$, whence
	$(-1)^{(|x|+k)(|y|+k)}=(-1)^{|x||y|}$ and $yD(x)=(-1)^{|y||x|}D(x)y$. The second
	identity is the graded Leibniz rule for $D$, whose sign $(-1)^{k|x|}$ is
	trivial, and the third statement is $(-1)^{|z|(|x|+k)}=(-1)^{|z||x|}$.
\end{proof}

\begin{theorem}\label{thm:unital-even}
	Let $k$ be even and let $(A,\cdot)$ be a unital graded commutative associative
	algebra. The assignment
	\[
		D\longmapsto[-,-]_D,
		\qquad
		[x,y]_D=x\,D(y)-D(x)\,y,
	\]
	is a bijection between the homogeneous derivations of $(A,\cdot)$ of degree $k$
	and the transposed brackets of degree $k$ on $(A,\cdot)$. Its inverse sends a
	bracket to $D=[1,-]$.
\end{theorem}

\begin{proof}
Let $D$ be a homogeneous derivation of degree $k$. By \Cref{lem:even-no-signs}
the displayed formula is the graded Wronskian
$xD(y)-(-1)^{(|x|+k)(|y|+k)}yD(x)$, and every identity below may be verified on
the sign-free expression.

The bracket has degree $k$, and
\[
        -(-1)^{(|x|+k)(|y|+k)}[y,x]_D
        =-(-1)^{|x||y|}\bigl(yD(x)-D(y)x\bigr)
        =-D(x)y+xD(y)=[x,y]_D ,
\]
which is the graded skew-symmetry of a bracket of degree $k$.

For the graded Jacobi identity, write $a=|x|$, $b=|y|$, $c=|z|$, and recall that
$D(w)$ and $D^{2}(w)$ have the same parity as $w$. Since $D$ is a derivation,
\[
\begin{aligned}
        D([x,y]_D)
        &=D\bigl(xD(y)\bigr)-D\bigl(D(x)y\bigr)\\
        &=D(x)D(y)+xD^{2}(y)-D^{2}(x)y-D(x)D(y)=xD^{2}(y)-D^{2}(x)y ,
\end{aligned}
\]
whence
\[
        [[x,y]_D,z]_D=[x,y]_D\,D(z)-D([x,y]_D)\,z
        =\bigl(xD(y)-D(x)y\bigr)D(z)-\bigl(xD^{2}(y)-D^{2}(x)y\bigr)z .
\]
Expanding this and its two cyclic permutations,
\[
\begin{aligned}
\relax[[x,y]_D,z]_D&=xD(y)D(z)-D(x)yD(z)-xD^{2}(y)z+D^{2}(x)yz,\\
[[y,z]_D,x]_D&=yD(z)D(x)-D(y)zD(x)-yD^{2}(z)x+D^{2}(y)zx,\\
[[z,x]_D,y]_D&=zD(x)D(y)-D(z)xD(y)-zD^{2}(x)y+D^{2}(z)xy.
\end{aligned}
\]
The graded Jacobi identity for a bracket of degree $k$ reads
\[
        (-1)^{ac}[[x,y]_D,z]_D+(-1)^{ab}[[y,z]_D,x]_D+(-1)^{bc}[[z,x]_D,y]_D=0 ,
\]
so we reorder every monomial of the second and third rows into the order
$x,y,z$. Each monomial of the second row is a product of three factors of
parities $b$, $c$ and $a$, in this order, and bringing the last one to the front
costs $(-1)^{a(b+c)}$; since $(-1)^{ab}(-1)^{a(b+c)}=(-1)^{ac}$, the second row
acquires the same prefactor as the first. Each monomial of the third row is a
product of factors of parities $c$, $a$ and $b$, and taking the first one to the
back costs $(-1)^{c(a+b)}$, with $(-1)^{bc}(-1)^{c(a+b)}=(-1)^{ac}$. Therefore
all three summands carry the common factor $(-1)^{ac}$, and
\[
\begin{aligned}
(-1)^{ac}[[x,y]_D,z]_D&=(-1)^{ac}\bigl(\phantom{{}-{}}xD(y)D(z)-D(x)yD(z)-xD^{2}(y)z+D^{2}(x)yz\bigr),\\
(-1)^{ab}[[y,z]_D,x]_D&=(-1)^{ac}\bigl(\phantom{{}-{}}D(x)yD(z)-D(x)D(y)z-xyD^{2}(z)+xD^{2}(y)z\bigr),\\
(-1)^{bc}[[z,x]_D,y]_D&=(-1)^{ac}\bigl(-xD(y)D(z)+D(x)D(y)z-D^{2}(x)yz+xyD^{2}(z)\bigr).
\end{aligned}
\]
Six monomials occur on the right, each of them twice and with opposite signs:
$xD(y)D(z)$ and $D^{2}(x)yz$ in the first and third rows, $D(x)yD(z)$ and
$xD^{2}(y)z$ in the first and second, $D(x)D(y)z$ and $xyD^{2}(z)$ in the second
and third. Hence the sum vanishes.
	
	For the transposed Leibniz rule, \Cref{lem:even-no-signs} gives
	\[
		[zx,y]_D=zxD(y)-D(z)xy-zD(x)y ,
		\qquad
		[x,zy]_D=xD(z)y+xzD(y)-D(x)zy .
	\]
	Multiplying the second identity by $(-1)^{|z||x|}$ turns $xD(z)$ into $D(z)x$,
	$xz$ into $zx$ and $D(x)z$ into $zD(x)$, so that the two terms in $D(z)$ cancel
	and
	\[
		[zx,y]_D+(-1)^{|z||x|}[x,zy]_D
		=2zxD(y)-2zD(x)y=2z[x,y]_D .
	\]
	
	Conversely, let $[-,-]$ be a transposed bracket of degree $k$ on $(A,\cdot)$ and
	set $D=[1,-]$, homogeneous of degree $k$. Since $k$ is even, the right hand side
	of \eqref{eq:parity-split} is $2D(zy)$, so
	\[
		D(zy)=zD(y)+(-1)^{|z||y|}yD(z)=zD(y)+D(z)y ,
	\]
	the second equality by \Cref{lem:even-no-signs}. Hence $D$ is a homogeneous
	derivation of degree $k$. Substituting back in \Cref{lem:unital-D} gives
	$[z,y]=2zD(y)-D(zy)=zD(y)-D(z)y=[z,y]_D$. Finally a derivation satisfies
	$D(1)=0$, so $[1,y]_D=D(y)$ and the two assignments are mutually inverse.
\end{proof}

\begin{theorem}\label{thm:unital-odd}
	Let $k$ be odd and let $(A,\cdot)$ be a unital graded commutative associative
	algebra. The assignment
	\[
		\theta\longmapsto[-,-]_\theta,
		\qquad
		[x,y]_\theta=(-1)^{|y|}(xy)\theta,
	\]
	is a bijection between $A^{k}$ and the transposed brackets of degree $k$ on
	$(A,\cdot)$. Its inverse sends a bracket to $\theta=[1,1]$.
\end{theorem}

\begin{proof}
	Suppose first that $[-,-]$ is a transposed bracket of degree $k$ on $(A,\cdot)$,
	and set $D=[1,-]$ and $\theta:=D(1)=[1,1]$. Since the bracket has degree $k$ and
	$|1|=0$, we have $\theta\in A^{k}$.
	
	As $k$ is odd, the right hand side of \eqref{eq:parity-split} vanishes, and the
	lemma reduces to
	\begin{equation}\label{eq:odd-relation}
		zD(y)+(-1)^{(|z|+k)(|y|+k)}yD(z)=0 .
	\end{equation}
	Take $z=1$ in \eqref{eq:odd-relation}. The exponent becomes
	$k(|y|+k)\equiv k|y|+k\equiv|y|+1$, because $k$ is odd, so
	\[
		D(y)-(-1)^{|y|}\,y\theta=0,
		\qquad\text{that is}\qquad
		D(y)=(-1)^{|y|}\,y\theta .
	\]
	Substituting this into \Cref{lem:unital-D}, and using
	$zD(y)=(-1)^{|y|}(zy)\theta$, $D(zy)=(-1)^{|z|+|y|}(zy)\theta$ and
	$(-1)^{k|z|}=(-1)^{|z|}$, we obtain
	\[
		[z,y]=2(-1)^{|y|}(zy)\theta-(-1)^{|z|}(-1)^{|z|+|y|}(zy)\theta
		=(-1)^{|y|}(zy)\theta ,
	\]
	which is the asserted formula.
	
	Conversely, let $\theta\in A^{k}$ and let $[x,y]_\theta=(-1)^{|y|}(xy)\theta$.
	This is homogeneous of degree $k$, since $|x|+|y|+|\theta|=|x|+|y|+k$. We check
	the three required identities. Write $a=|x|$ and $b=|y|$.
	
	Since $|\theta|=k$ is odd, graded commutativity gives
	$\theta\theta=(-1)^{k^2}\theta\theta=-\theta\theta$, so $2\theta^2=0$ and
	therefore
	\begin{equation}\label{eq:theta-square}
		\theta^2=0 ,
	\end{equation}
	the characteristic of $\mathbb K$ being different from two.
	
	For the graded skew-symmetry, $yx=(-1)^{ab}xy$ gives
	\[
		-(-1)^{(a+k)(b+k)}[y,x]_\theta
		=-(-1)^{ab+a+b+1}(-1)^{a}(-1)^{ab}(xy)\theta
		=-(-1)^{b+1}(xy)\theta
		=[x,y]_\theta ,
	\]
	where $(a+k)(b+k)\equiv ab+a+b+1$ because $k$ is odd.
	
	For the graded Jacobi identity it suffices to observe that every iterated
	bracket vanishes. Indeed
	\[
		[[x,y]_\theta,z]_\theta
		=(-1)^{|z|+b}\bigl((xy)\theta z\bigr)\theta
		=(-1)^{|z|+b+k|z|}(xyz)\,\theta^2=0
	\]
	by \eqref{eq:theta-square}, and likewise for the other two summands.
	
	For the transposed Leibniz rule, on the one hand
	$[zx,y]_\theta=(-1)^{b}(zxy)\theta$. On the other hand, using
	$(-1)^{k|z|}=(-1)^{|z|}$ and $xzy=(-1)^{a|z|}zxy$,
	\[
		(-1)^{|z|(a+k)}[x,zy]_\theta
		=(-1)^{|z|a+|z|}(-1)^{|z|+b}(xzy)\theta
		=(-1)^{b}(zxy)\theta .
	\]
	Their sum is $2(-1)^{b}(zxy)\theta=2z[x,y]_\theta$.
	
	Finally $[1,1]_\theta=\theta$, so the two assignments are mutually inverse.
\end{proof}

\begin{remark}\label{rmk:jacobi-automatic}
	Neither \Cref{thm:unital-even} nor \Cref{thm:unital-odd} uses the graded Jacobi
	identity in the direction which starts from a transposed bracket:
	\Cref{lem:unital-D} uses only \eqref{eq:graded-transposed-leibniz}, and
	\Cref{lem:parity-split} adds only graded skew-symmetry. Over a unital base,
	therefore, a bilinear map of degree $k$ which is graded skew-symmetric and
	satisfies \eqref{eq:graded-transposed-leibniz} satisfies the graded Jacobi
	identity as well.
\end{remark}

\begin{corollary}\label{cor:rigidity}
	Let $k$ be odd, and let $(A,\cdot)$ be a unital graded commutative associative
	algebra with $A^{k}=0$. Then the only transposed bracket of degree $k$ on
	$(A,\cdot)$ is the zero bracket. This is the case, for every odd $k<0$, whenever
	$A$ is non-negatively graded.
\end{corollary}

\begin{example}\label{ex:theta}
	The hypothesis $A^{k}=0$ cannot be removed. Let $k=-1$ and let
	$A=\Lambda(\theta)$ be the exterior algebra on one generator placed in degree
	$-1$, so that $A^0=\mathbb K\cdot 1$, $A^{-1}=\mathbb K\theta$ and
	$\theta^2=0$. Then $[x,y]=(-1)^{|y|}(xy)\theta$ is a non-zero transposed bracket
	of degree $-1$ on $(A,\cdot)$, with $[1,1]=\theta$.
\end{example}

We collect three comments on the role played by the parity of $k$.
\begin{enumerate}
\item \Cref{thm:unital-even} does not extend to odd degrees. The reason is
visible in \eqref{eq:wronskian-nosign}: for $k$ odd, $|D(x)|$ and $|x|$ have
opposite parities, that rewriting is unavailable, and in the proof of the
transposed Leibniz rule the two terms in $D(z)$ add instead of cancelling, the
discrepancy being $(1-(-1)^{k})xyD(z)$. Concretely, let
$\mathcal G=\Lambda(\xi)$ with $|\xi|=1$ and let $D$ be the derivation of degree
$-1$ with $D(\xi)=1$. Then $[x,y]_D=xD(y)-(-1)^{(|x|-1)(|y|-1)}yD(x)$ gives
$[1,\xi]_D=1$, whereas $2\xi[1,\xi]_D=2\xi$ and $[\xi,\xi]_D-[1,\xi\xi]_D=0$,
so \eqref{eq:graded-transposed-leibniz} fails.

\item The ternary constructions of the ungraded theory are obstructed as well.
If $(L,\cdot,[-,-])$ is a transposed Poisson algebra and $D$ is a derivation of
both operations, then
\[
        [x,y,z]_D=D(x)[y,z]+D(y)[z,x]+D(z)[x,y]
\]
makes $L$ into a $3$-Lie algebra \cite{Bai2023}, and every step of that argument
passes through the cyclic identity \eqref{eq:cyclic-basic}. For $k$ even the
product, the bracket and $D$ are all even operations, so the ungraded identities
transfer verbatim once the Koszul signs are inserted. For $k$ odd the cyclic
identity is no longer available, by \Cref{prop:master-identity}, and the
construction fails at its first step. We do not pursue the graded ternary theory
here.

\item The parity of $k$ decides the shape of the classification, not its
outcome. The classifying object is $\operatorname{Der}^k(A)$ in one case and
$A^k$ in the other, and both depend on $k$ itself. On $\Lambda(\xi)$ with
$|\xi|=1$ the transposed brackets of degree $-1$ reduce to zero while those of
degree $+1$ form a line, and on $\Lambda(\theta,\xi)$ with $|\theta|=-1$ and
$|\xi|=1$ the transposed brackets of degree $0$ form a plane while those of
degree $-2$ form a line.
\end{enumerate}

\section{Transposed Gerstenhaber algebras and Lie--Rinehart rigidity}\label{sec:gerstenhaber}

We now specialise to $k=-1$, the degree of the Gerstenhaber bracket, and
therefore the one in which the odd-degree rigidity of \Cref{sec:graded} meets
the geometric examples.

\begin{definition}
	A \emph{transposed Gerstenhaber algebra} is a graded transposed Poisson algebra
	of degree $-1$. Thus it is a triple $(\mathcal G,\cdot,[-,-])$ such that
	$(\mathcal G,\cdot)$ is graded commutative associative, $[-,-]$ is a graded Lie
	bracket of degree $-1$, and
	\[
		2z[x,y]=[zx,y]+(-1)^{|z|(|x|-1)}[x,zy]
	\]
	for homogeneous $x,y,z\in\mathcal G$.
\end{definition}

\begin{remark}
	This is not merely the super version of a transposed Poisson algebra. The
	super version corresponds to keeping only the parity of the grading and to an
	even bracket. Here the bracket has degree $-1$, hence it is odd. This
	distinction is essential in the signs and in the identities which follow.
\end{remark}

Specialising \Cref{thm:unital-odd} to $k=-1$ describes all transposed
Gerstenhaber brackets over a unital base.

\begin{theorem}\label{thm:transposed-gerstenhaber-unital}
	Let $(\mathcal G,\cdot)$ be a unital graded commutative associative algebra. The
	transposed Gerstenhaber brackets on $(\mathcal G,\cdot)$ are exactly the
	brackets
	\[
		[x,y]_\theta=(-1)^{|y|}(xy)\theta,
		\qquad \theta\in\mathcal G^{-1},
	\]
	and $\theta$ is recovered from the bracket as $\theta=[1,1]$. In particular, if
	$\mathcal G$ is non-negatively graded, then the only transposed Gerstenhaber
	bracket on $(\mathcal G,\cdot)$ is the zero bracket.
\end{theorem}

This settles the second half of \Cref{pb:LR} for the graded commutative algebra
$\Lambda_A L$, which is non-negatively graded and unital by
\Cref{rmk:unital-convention}. The first half is settled in
\Cref{sec:LR-obstruction} below. The theory is not empty, as \Cref{ex:theta}
shows, but its non-zero examples live in negative degrees, and none of them is
of geometric origin.

\begin{example}
	Every graded commutative associative algebra with the zero bracket is a
	transposed Gerstenhaber algebra, and by
	\Cref{thm:transposed-gerstenhaber-unital} these are the only examples over a
	unital base which is non-negatively graded.
\end{example}

We now fix both operations, and ask when a given Gerstenhaber algebra happens to
be transposed. Let $(\mathcal G,\cdot,[-,-])$ be a Gerstenhaber algebra. There is no reason for
the Gerstenhaber Leibniz rule to imply the transposed Leibniz rule. We measure
the failure of the latter by the trilinear map
\[
	\Theta_{\mathcal G}(x,y,z)
	=
	[zx,y]
	+(-1)^{|z|(|x|-1)}[x,zy]
	-2z[x,y],
\]
defined on homogeneous elements and extended by linearity.

\begin{proposition}
	A Gerstenhaber algebra $(\mathcal G,\cdot,[-,-])$ is a transposed
	Gerstenhaber algebra if and only if
	\[
		\Theta_{\mathcal G}=0 .
	\]
\end{proposition}

\begin{proof}
	This is exactly the transposed Leibniz rule in degree $-1$.
\end{proof}

The usefulness of $\Theta_{\mathcal G}$ is that it can be computed by means of
the ordinary Gerstenhaber rule alone, and that the resulting expression is
remarkably small.

\begin{lemma}\label{lem:theta-formula}
	For all homogeneous $x,y,z$ in a Gerstenhaber algebra $\mathcal G$,
	\[
		\Theta_{\mathcal G}(x,y,z)
		=
		(-1)^{|x|(|y|-1)}[z,y]\,x
		+
		(-1)^{|z|(|x|-1)}[x,z]\,y .
	\]
\end{lemma}

\begin{proof}
	By \Cref{eq:gerst-leibniz-first} applied to $[zx,y]$ and by
	\Cref{eq:graded-poisson-leibniz} applied to $[x,zy]$,
	\[
		[zx,y]=(-1)^{|x|(|y|-1)}[z,y]x+z[x,y],
		\qquad
		[x,zy]=[x,z]y+(-1)^{(|x|-1)|z|}z[x,y].
	\]
	Substituting these two expressions in the definition of $\Theta_{\mathcal G}$,
	the three multiples of $z[x,y]$ combine into
	$z[x,y]+z[x,y]-2z[x,y]=0$.
\end{proof}

The formula has an immediate consequence, which is the mechanism behind all the
rigidity statements of this section.

\begin{theorem}\label{thm:unital-gerstenhaber-rigid}
	Let $(\mathcal G,\cdot,[-,-])$ be a Gerstenhaber algebra whose underlying
	graded commutative algebra is unital. Then $\mathcal G$ is a transposed
	Gerstenhaber algebra if and only if $[-,-]=0$.
\end{theorem}

\begin{proof}
	In a Gerstenhaber algebra with unit one has $[1,w]=0$ for every $w$, as follows
	from \Cref{eq:gerst-leibniz-first} with $y=z=1$. Assume
	$\Theta_{\mathcal G}=0$ and evaluate the formula in \Cref{lem:theta-formula} at $x=1$:
	\[
		0=\Theta_{\mathcal G}(1,y,z)=[z,y]+[1,z]\,y=[z,y].
	\]
	Hence the bracket vanishes. The converse is clear.
\end{proof}

In the next subsection we apply this to the Gerstenhaber algebra $\Lambda_A L$
associated with a Lie--Rinehart pair, and then we isolate the part of the
conclusion whose proof stays in degree one and never evaluates at the unit.

\subsection{The transposed obstruction on \texorpdfstring{$\Lambda_A L$}{Lambda A L}}
\label{sec:LR-obstruction}

Throughout this subsection $(A,L)$ is a Lie--Rinehart algebra over $(R,A)$, and
$\Lambda_A L$ is endowed with its Schouten--Rinehart bracket.

\begin{theorem}\label{thm:LR-complete}
	$(\Lambda_A L,\wedge,[-,-]_{SR})$ is a transposed Gerstenhaber algebra if and
	only if the anchor of $(A,L)$ vanishes and $L$ is abelian.
\end{theorem}

\begin{proof}
	The graded commutative algebra $\Lambda_A L$ is unital, so
	\Cref{thm:unital-gerstenhaber-rigid} applies and transposedness is equivalent to
	$[-,-]_{SR}=0$. Since $[X,a]_{SR}=X(a)$ and $[X,Y]_{SR}=[X,Y]$ for $a\in A$ and
	$X,Y\in L$, this is equivalent to $\rho=0$ and $[L,L]=0$. Conversely, if the
	anchor and the bracket of $L$ both vanish, then $[-,-]_{SR}=0$ and the
	transposed identity is immediate.
\end{proof}

The result admits a reformulation which does not refer to a fixed Lie--Rinehart
structure. For $A$ unital and $L$ an $A$-module, Huebschmann's correspondence is
a bijection between the Lie--Rinehart structures on $L$ over $A$ and the
Gerstenhaber brackets of degree $-1$ on $(\Lambda_A L,\wedge)$, under which the
anchor and the bracket are recovered as $\rho(X)(a)=[X,a]_{SR}$ and
$[X,Y]=[X,Y]_{SR}$. See \cite{Huebschmann1990}, and \cite[Thm~6.3]{Gatse2026} for a
self-contained derivation-based account.

\begin{corollary}\label{cor:LR-correspondence}
Let $A$ be a unital commutative $\mathbb K$-algebra and let $L$ be an
$A$-module. Under this correspondence, the transposed Gerstenhaber brackets on
$(\Lambda_A L,\wedge)$ correspond exactly to the trivial Lie--Rinehart
structure, in which both the anchor and the bracket of $L$ vanish.
\end{corollary}

\begin{proof}
The algebra $\Lambda_A L$ is unital and non-negatively graded, so by
\Cref{thm:transposed-gerstenhaber-unital} the only transposed Gerstenhaber
bracket on it is the zero bracket. Under the correspondence, this bracket has
$\rho(X)(a)=0$ and $[X,Y]=0$.
\end{proof}

In particular, no Gerstenhaber bracket on $\Lambda_A L$ escapes
\Cref{thm:LR-complete} by failing to be of Schouten--Rinehart type.

It is instructive to see the two halves of \Cref{thm:LR-complete} appear
separately in low degrees. The first argument evaluates at the unit, the second
does not.

\begin{proposition}\label{prop:anchor-vanishes}
	If the Gerstenhaber algebra $\Lambda_A L$ is transposed Gerstenhaber, then the
	anchor of $(A,L)$ is zero. Equivalently,
	\[
		X(a)=0
	\]
	for all $X\in L$ and $a\in A$.
\end{proposition}

\begin{proof}
	Let $a\in A$ and $X\in L$. Apply the transposed Gerstenhaber identity in
	$\Lambda_A L$ with $x=1$, $y=X$ and $z=a$. Since $|1|=|a|=0$ and $|X|=1$, we get
	\[
		2a[1,X]_{SR}
		=
		[a,X]_{SR}+[1,aX]_{SR}.
	\]
	Now $[1,X]_{SR}=0$ and $[1,aX]_{SR}=0$, while
	\[
		[a,X]_{SR}=-[X,a]_{SR}=-X(a).
	\]
	Hence $X(a)=0$. Since $a$ and $X$ were arbitrary, the anchor vanishes.
\end{proof}

Thus the transposed identity immediately kills the geometric part of the
Lie--Rinehart structure. In particular, the usual Gerstenhaber algebra of
multivector fields on a smooth manifold is not transposed Gerstenhaber unless
all vector fields act trivially on functions.

\begin{example}
	Let $M$ be a smooth manifold and let
	\[
		A=C^\infty(M),\qquad L=\mathfrak X(M).
	\]
	Then $\Lambda_A L$ is the Gerstenhaber algebra of multivector fields with the
	Schouten bracket, and by \Cref{thm:LR-complete} it is not transposed
	Gerstenhaber, unless $M$ is a point.
\end{example}

\Cref{thm:LR-complete} closes \Cref{pb:LR}, and it does so by evaluating the
transposed identity at the unit. We now isolate the part of that conclusion
which is reached without ever leaving degree one. The results below are
necessary conditions and are weaker than \Cref{thm:LR-complete}. Their interest
is that their proofs use only the bracket of $L$ and the exterior product, so
they remain available whenever $\Lambda^\bullet L$ is defined and the evaluation
at a unit is not.

\begin{proposition}\label{prop:exterior-obstruction}
	Assume that $\Lambda_A L$ is transposed Gerstenhaber. Then, for all
	$X,Y,Z\in L$,
	\[
		[Z,Y]\wedge X+[X,Z]\wedge Y=0
	\]
	in $\Lambda_A^2L$.
\end{proposition}

\begin{proof}
	Apply the transposed Gerstenhaber identity with $x=X$, $y=Y$ and $z=Z$.
	Since $|X|=|Y|=|Z|=1$, the sign is equal to $1$, and we obtain
	\[
		2Z\wedge [X,Y]
		=
		[Z\wedge X,Y]_{SR}+[X,Z\wedge Y]_{SR}.
	\]
	Using the Schouten--Rinehart bracket,
	\[
		[Z\wedge X,Y]_{SR}
		=
		[Z,Y]\wedge X-[X,Y]\wedge Z,
	\]
	and
	\[
		[X,Z\wedge Y]_{SR}
		=
		[X,Z]\wedge Y+Z\wedge [X,Y].
	\]
	Since $Z\wedge [X,Y]=-[X,Y]\wedge Z$, the right-hand side becomes
	\[
		[Z,Y]\wedge X+[X,Z]\wedge Y-2[X,Y]\wedge Z.
	\]
	The left-hand side is
	\[
		2Z\wedge [X,Y]=-2[X,Y]\wedge Z.
	\]
	After cancellation we obtain
	\[
		[Z,Y]\wedge X+[X,Z]\wedge Y=0,
	\]
	as claimed.
\end{proof}

\begin{definition}
	The \emph{exterior annihilator} of $L$ is the $A$-submodule
	\[
		\operatorname{Ann}_{\wedge}(L)
		=
		\{\omega\in L\mid \omega\wedge U=0
		\text{ for every } U\in L\}.
	\]
	We shall say that $L$ is \emph{exteriorly non-degenerate} if
	\[
		\operatorname{Ann}_{\wedge}(L)=0.
	\]
\end{definition}

For free modules of rank at least two this condition is the expected one: no
non-zero vector wedges trivially with every other vector. For general modules,
however, $\operatorname{Ann}_{\wedge}(L)$ may be non-zero; this is precisely
where exterior degeneracies can hide non-abelian brackets.

\begin{proposition}\label{prop:bracket-exterior-annihilator}
	Assume that $\Lambda_A L$ is transposed Gerstenhaber. Then, for all
	$X,Y\in L$,
	\[
		[X,Y]\wedge X=0
		\qquad\text{and}\qquad
		[X,Y]\wedge Y=0 .
	\]
	In particular, if $[X,Y]\in\operatorname{Ann}_{\wedge}(L)$ and $L$ is
	exteriorly non-degenerate, then $[X,Y]=0$.
\end{proposition}

\begin{proof}
	By \Cref{prop:exterior-obstruction}, for all $X,Y,Z\in L$,
	\[
		[Z,Y]\wedge X+[X,Z]\wedge Y=0.
	\]
	Putting $Z=X$ gives $[X,Y]\wedge X+[X,X]\wedge Y=0$, and $[X,X]=0$. Putting
	$Z=Y$ gives $[Y,Y]\wedge X+[X,Y]\wedge Y=0$, and again $[Y,Y]=0$. The final
	assertion follows directly from the definition of
	$\operatorname{Ann}_{\wedge}(L)$.
\end{proof}

\begin{corollary}\label{cor:field-case}
	Let $\mathfrak g$ be a Lie algebra over a field $\mathbb K$. Then the
	Schouten Gerstenhaber algebra
	\[
		\bigl(\Lambda \mathfrak g,\wedge,[-,-]_{S}\bigr)
	\]
	is transposed Gerstenhaber if and only if $\mathfrak g$ is abelian.
\end{corollary}

\begin{proof}
	If $\mathfrak g$ is abelian, the Schouten bracket is zero, hence the
	transposed identity is immediate.
	
	Conversely, assume that $\Lambda\mathfrak g$ is transposed Gerstenhaber, and
	let $X,Y\in\mathfrak g$. By \Cref{prop:bracket-exterior-annihilator},
	\[
		[X,Y]\wedge X=0
		\qquad\text{and}\qquad
		[X,Y]\wedge Y=0 .
	\]
	If $X$ and $Y$ are linearly independent, these two identities force
	$[X,Y]$ to be proportional both to $X$ and to $Y$, hence $[X,Y]=0$.
	If $X$ and $Y$ are linearly dependent, then $[X,Y]=0$ by skew-symmetry.
	Thus $\mathfrak g$ is abelian.
\end{proof}

The same conclusion follows from \Cref{thm:LR-complete}. The point of the
argument above is that it uses only elements of degree one.

\begin{remark}
	Over a general commutative algebra $A$ the conclusion of \Cref{cor:field-case}
	does not follow from the exterior obstruction alone. The reason is that, for an
	arbitrary $A$-module, the vanishing of $\omega\wedge U$ for many $U$'s need not
	have the same linear-algebraic meaning as over a field. In the degree one
	argument non-abelianity can survive only through degeneracies of the exterior
	algebra, and the natural receptacle for such degeneracies is
	$\operatorname{Ann}_{\wedge}(L)$.
\end{remark}

\begin{corollary}
	Assume that $\Lambda_A L$ is transposed Gerstenhaber. If
	$[L,L]\subseteq \operatorname{Ann}_{\wedge}(L)$, then the exterior
	obstruction of \Cref{prop:exterior-obstruction} vanishes identically.
\end{corollary}

\begin{proof}
	If $[L,L]\subseteq \operatorname{Ann}_{\wedge}(L)$, then each term
	$[Z,Y]\wedge X$ and $[X,Z]\wedge Y$ is zero.
\end{proof}

Thus the obstruction has two different aspects. Evaluated at the unit it is
absolute: the whole Schouten--Rinehart bracket must vanish, anchor included.
Confined to degree one it is exterior, and it forces the Lie bracket into the
exterior degeneracy of $L$. The second form is the one which has a chance of
surviving outside the hypotheses of \Cref{rmk:unital-convention}, see
\Cref{pb:nonunital}.

\subsection{The gauge Lie--Rinehart algebra}

Let $E$ be an $A$-module. We denote by $\mathfrak D(E)$ the gauge
Lie--Rinehart algebra of $E$, namely the $R$-module of $R$-linear maps
$\Delta:E\to E$ for which there exists a derivation
$\sigma_\Delta\in \operatorname{Der}_R(A)$ satisfying
\[
	\Delta(ae)=a\Delta(e)+\sigma_\Delta(a)e,
	\qquad a\in A,\ e\in E.
\]
The derivation $\sigma_\Delta$ is called the symbol of $\Delta$. The bracket
on $\mathfrak D(E)$ is the commutator
\[
	[\Delta,\nabla]=\Delta\circ\nabla-\nabla\circ\Delta,
\]
and the anchor is the symbol map
\[
	\sigma:\mathfrak D(E)\longrightarrow \operatorname{Der}_R(A),
	\qquad
	\Delta\longmapsto \sigma_\Delta .
\]
Thus $\mathfrak D(E)$ is the algebraic analogue of the Atiyah algebroid of a
vector bundle. In the smooth case, the derivations of $E$ in the above sense are
exactly the sections of a Lie algebroid over $M$, with the commutator bracket
and the symbol as anchor \cite[Thm.~1.4]{KosmannSchwarzbachMackenzie2002}.

\begin{proposition}\label{prop:gauge-anchor}
	Let $E$ be an $A$-module. If the canonical Gerstenhaber algebra
	$\Lambda_A\mathfrak D(E)$ with the Schouten--Rinehart bracket is transposed
	Gerstenhaber, then every gauge derivation of $E$ has zero symbol, so that
	\[
		\mathfrak D(E)=\operatorname{End}_A(E),
	\]
	and moreover $\operatorname{End}_A(E)$ is commutative.
\end{proposition}

\begin{proof}
	The anchor of the gauge Lie--Rinehart algebra is the symbol map
	$\sigma:\mathfrak D(E)\to \operatorname{Der}_R(A)$, and its bracket is the
	commutator. By \Cref{thm:LR-complete}, transposedness forces both to vanish.
	The condition $\sigma_\Delta=0$ is exactly $A$-linearity of $\Delta$, so the
	surviving gauge operators are the elements of $\operatorname{End}_A(E)$, and the
	vanishing of the bracket says that they commute.
\end{proof}

\begin{example}
	Let $E=A^n$. Then
	\[
		\operatorname{End}_A(E)\simeq M_n(A).
	\]
	For $n\geq 2$ this algebra is noncommutative as soon as $A$ is nonzero. Hence
	the canonical Schouten--Rinehart Gerstenhaber algebra associated with the gauge
	Lie--Rinehart algebra of $A^n$ is never transposed Gerstenhaber.
\end{example}

\begin{remark}
	\Cref{prop:gauge-anchor} should be read as a gauge-theoretic form of the
	anchor obstruction. In geometric language, the Atiyah algebroid has an exact
	sequence whose quotient is the tangent directions of the base. Transposedness
	forces this quotient part to disappear, and then it also forces the remaining
	vertical part to be abelian.
\end{remark}

\section{The geometry of transposedness}\label{sec:geometry}

The transposed Leibniz rule is algebraic in nature. Nevertheless, since many
geometric structures are encoded by algebras of functions, it is natural to ask
what kind of geometry it defines.

We start with the most direct translation. Let $M$ be a smooth manifold. A
\emph{transposed Poisson structure} on $M$ is a Lie bracket $[-,-]$ on
$C^\infty(M)$ such that
\[
	2h[f,g]=[hf,g]+[f,hg]
\]
for all $f,g,h\in C^\infty(M)$. Equivalently,
$(C^\infty(M),\cdot,[-,-])$ is a transposed Poisson algebra.

Since $C^\infty(M)$ is unital, this condition is very restrictive. Indeed,
every such bracket is determined by the derivation
\[
	D(f)=[1,f].
\]
Thus there is a unique vector field $X\in\mathfrak X(M)$ such that
$D=X$, and
\[
	[f,g]=fX(g)-gX(f).
\]
Conversely, every vector field $X$ defines a transposed Poisson structure by
this formula.

\begin{proposition}\label{prop:transposed-manifold}
	Transposed Poisson structures on a smooth manifold $M$ are in one-to-one
	correspondence with vector fields on $M$.
\end{proposition}

\begin{proof}
	Let $[-,-]$ be a transposed Poisson bracket on $C^\infty(M)$, and set
	$D(f)=[1,f]$. Since $C^\infty(M)$ is unital, the standard unital argument
	for transposed Poisson algebras gives
	\[
		[f,g]=fD(g)-gD(f).
	\]
	Moreover, $D$ is a derivation of $C^\infty(M)$, hence a vector field on
	$M$. This proves that every transposed Poisson bracket is of the displayed
	form. The converse follows by a direct verification: for any vector field $X$ of $M$, the bracket
	\[
		[f,g]_X=fX(g)-gX(f)
	\]
	is skew-symmetric, satisfies the Jacobi identity, and obeys the transposed
	Leibniz rule.
\end{proof}

This shows that the word ``Poisson'' in this context should not be understood
in the usual geometric sense. A Poisson manifold is governed by a bivector
$\pi\in\Gamma(\Lambda^2TM)$. A transposed Poisson structure, in the sense
above, is governed by a vector field.

The same formula is familiar from Jacobi geometry. Recall that a Jacobi bracket
on the trivial line bundle is locally described by a pair
$(\Lambda,E)$, where $\Lambda$ is a bivector and $E$ is a vector field,
through
\[
	\{f,g\}_{(\Lambda,E)}
	=
	\Lambda(df,dg)+fE(g)-gE(f).
\]
Thus a transposed Poisson structure is precisely the special Jacobi bracket
corresponding to
\[
	\Lambda=0,\qquad E=X.
\]
In other words, transposed Poisson manifolds are not Poisson manifolds in
disguise. They form the purely vector-field part of Jacobi geometry.

\Cref{prop:transposed-manifold} extends from functions to the exterior algebra
of an arbitrary vector bundle, provided the degree of the bracket is zero. The
answer involves the same gauge algebroid which the degree $-1$ theory
annihilates.

\begin{theorem}\label{thm:degree-zero-bundle}
	Let $E\to M$ be a smooth vector bundle and let
	$\mathcal A=\Gamma(\Lambda^\bullet E)$ with the exterior product.
	\begin{enumerate}
		\item For every $k\geq1$, the only transposed bracket of degree $-k$ on
		      $\mathcal A$ is the zero bracket.
		\item For $k=0$, the assignment
		      \[
		      	\Delta\longmapsto[-,-]_\Delta,
		      	\qquad
		      	[P,Q]_\Delta=P\wedge \Delta(Q)-\Delta(P)\wedge Q,
		      \]
		      is a bijection between $\Gamma(\mathfrak D(E))$ and the transposed brackets of
		      degree $0$ on $\mathcal A$.
	\end{enumerate}
\end{theorem}

\begin{proof}
	The algebra $\mathcal A$ is unital and non-negatively graded, and it is
	generated over $\mathbb R$ by $\mathcal A^0=C^\infty(M)$ and
	$\mathcal A^1=\Gamma(E)$. Here we use that a homogeneous derivation $D$ of
	$\mathcal A$ is a local operator: if $s$ vanishes on an open set $U$ and
	$p\in U$, choosing $\chi$ supported in $U$ with $\chi(p)=1$ gives $\chi s=0$,
	hence $0=D(\chi s)=D(\chi)s+\chi D(s)$, which at $p$ reads $D(s)(p)=0$.
	Locality and a partition of unity reduce the determination of $D$ to
	decomposable sections, which are products of elements of $\mathcal A^0$ and
	$\mathcal A^1$.
	
	If $-k$ is odd, assertion (1) is \Cref{cor:rigidity}. If $-k$ is even,
	\Cref{thm:unital-even} identifies the transposed brackets with the homogeneous
	derivations of $\mathcal A$ of degree $-k$. For $k\geq2$ such a derivation sends
	$\mathcal A^0$ into $\mathcal A^{-k}=0$ and $\mathcal A^1$ into
	$\mathcal A^{1-k}=0$, hence it vanishes on generators and is zero.
	
	For $k=0$, a homogeneous derivation $D$ of degree zero is determined by its
	restrictions to $\mathcal A^0$ and $\mathcal A^1$. The restriction to
	$\mathcal A^0$ is a derivation of $C^\infty(M)$, that is a vector field $V$,
	and the Leibniz rule for $f\in\mathcal A^0$ and $s\in\mathcal A^1$ reads
	$D(fs)=V(f)s+fD(s)$. Thus $D|_{\mathcal A^1}$ is a gauge derivation of $E$ with
	symbol $V$, and conversely every element of $\Gamma(\mathfrak D(E))$ extends
	uniquely to a degree zero derivation of $\Gamma(\Lambda^\bullet E)$. Assertion
	(2) now follows from \Cref{lem:even-no-signs} and \Cref{thm:unital-even}, the
	bracket reading
	$P\wedge\Delta(Q)-\Delta(P)\wedge Q$ in degree zero.
\end{proof}

\begin{example}
	For $E=TM$ one obtains the algebra of multivector fields. Its transposed
	brackets of degree $0$ are exactly the gauge derivations of $TM$, while by
	\Cref{thm:LR-complete} its Schouten bracket, which has degree $-1$, admits none.
	For a Lie algebroid $A\to M$ and $X\in\Gamma(A)$, the operator
	$\Delta=[X,-]_{SR}$ is a gauge derivation of $A$ with symbol $\rho(X)$, and the
	corresponding degree zero bracket on $\Gamma(\Lambda^\bullet A)$ is
	$[P,Q]_X=P\wedge[X,Q]_{SR}-[X,P]_{SR}\wedge Q$.
\end{example}

Both statements point in the same direction. Transposedness in the graded
setting is governed by first-order operators, not by bivectors, and it is
compatible with the grading only when the bracket is even.

The distinction above is important for what follows. We shall not impose the
ordinary Poisson Leibniz rule and the transposed Leibniz rule on the same
bracket. Doing so would force the bracket to vanish. Rather, the point is to
ask whether transposedness can be transported along the geometric constructions
naturally attached to Poisson and Lie--Rinehart structures.

For instance, a Poisson algebra $(A,\{\,,\,\})$ determines a cotangent
Lie--Rinehart algebra on $\Omega^1(A)$, with anchor
\[
	\rho(df)=\{f,-\}
\]
and bracket characterised by
\[
	[[df,dg]]=d\{f,g\},
\]
see \cite{Huebschmann1990}.
The bracket obtained in this way is Koszul's
\cite{Koszul1985}, and the resulting passage from a Poisson algebra to the
Gerstenhaber algebra $\Omega^\bullet(A)$ is the one studied in
\cite{Kosmann-Schwarzbach1996}. The construction is the algebraic analogue of
the cotangent Lie algebroid of a Poisson manifold \cite{Vaisman1994}. The question we shall consider is not whether
$(A,\cdot,\{\,,\,\})$ is both Poisson and transposed, but whether the
Gerstenhaber algebra canonically associated with this cotangent
Lie--Rinehart algebra is transposed.

The answer is rigid. If the Schouten--Rinehart Gerstenhaber algebra
\[
	\Lambda_A\Omega^1(A)
\]
associated with the cotangent Lie--Rinehart algebra is transposed
Gerstenhaber, then by \Cref{thm:LR-complete} its anchor must vanish. In
geometric terms, for a Poisson manifold $(M,\pi)$, this means that the cotangent
anchor
\[
	\pi^\sharp:T^*M\longrightarrow TM
\]
must be zero. Hence $\pi=0$.

Thus transposedness is transportable to geometry only in a limited sense. On
the algebra of functions it produces vector-field-type Jacobi brackets, and on
exterior algebras of vector bundles it produces gauge derivations. Along the
canonical Poisson--cotangent--Gerstenhaber construction, however, it becomes
a rigidity condition: it annihilates the geometric anchor.

\section{Transposed \texorpdfstring{$F$}{F}-algebroids}\label{sec:Falgebroids}

The preceding sections show that the Lie--Rinehart and Gerstenhaber route is
too rigid for transposedness: over a unital base the Schouten--Rinehart bracket
must vanish identically. The obstruction, however, is attached to the odd degree
of the Gerstenhaber bracket, not to the presence of an anchor. It is therefore
natural to look for a geometric framework in which the bracket stays in degree
zero, and the anchor is retained.

Transposed Poisson algebras are $F$-manifold algebras, and $F$-algebroids are
the algebroid-level version of $F$-manifold algebras. This section takes that
route. We define transposed $F$-algebroids, we classify those with an identity,
and we describe the others locally wherever the anchor does not vanish. We recall the definitions, following \cite{CruzMorales2023}.

\begin{definition}[\cite{CruzMorales2023,Hertling1999}]
	An \emph{$F$-manifold algebra} is a triple $(\mathfrak g,[-,-],\cdot)$ where
	$(\mathfrak g,\cdot)$ is a commutative associative algebra, $(\mathfrak
	g,[-,-])$ is a Lie algebra, and the Hertling--Manin relation
	\begin{equation}\label{eq:HM}
		P_{x\cdot y}(z,w)=x\cdot P_y(z,w)+y\cdot P_x(z,w)
	\end{equation}
	holds for all $x,y,z,w\in\mathfrak g$, where
	\begin{equation}\label{eq:Ptensor}
		P_x(y,z)=[x,y\cdot z]-[x,y]\cdot z-y\cdot[x,z].
	\end{equation}
\end{definition}

The tensor $P_x(y,z)$ measures the failure of $\operatorname{ad}_x$ to be a
derivation of the commutative product, so an $F$-manifold algebra is a Poisson
algebra up to the correction \eqref{eq:HM}. An \emph{$F$-manifold} is a
manifold $M$ with a $C^\infty(M)$-bilinear commutative associative product
$\bullet$ on $TM$ such that $(\mathfrak X(M),[-,-],\bullet)$ is an $F$-manifold
algebra \cite[Def.~2.4]{CruzMorales2023}.

\begin{definition}[{\cite[Def.~2.6]{CruzMorales2023}}]
	A \emph{commutative associative algebroid} is a vector bundle $A\to M$ equipped
	with a $C^\infty(M)$-bilinear, commutative and associative multiplication
	$\cdot_A$ on $\Gamma(A)$.
\end{definition}

\begin{definition}[{\cite[Def.~2.7]{CruzMorales2023}}]\label{def:Falgebroid}
	An \emph{$F$-algebroid} is a vector bundle $A\to M$ equipped with a
	multiplication $\cdot_A$, a skew-symmetric bracket $[-,-]_A$ on $\Gamma(A)$ and
	a bundle map $a_A:A\to TM$, called the anchor, such that
	$(A,[-,-]_A,a_A)$ is a Lie algebroid, $(A,\cdot_A)$ is a commutative associative
	algebroid, and $(\Gamma(A),[-,-]_A,\cdot_A)$ is an $F$-manifold algebra. A
	section $e\in\Gamma(A)$ is an \emph{identity} if $e\cdot_A X=X$ for all
	$X\in\Gamma(A)$ \cite[Def.~2.11]{CruzMorales2023}.
\end{definition}

We shall use throughout the Lie algebroid axiom in the form
\begin{equation}\label{eq:algebroid-leibniz}
	[X,fY]_A=f[X,Y]_A+a_A(X)(f)\,Y,
	\qquad X,Y\in\Gamma(A),\ f\in C^\infty(M),
\end{equation}
together with the $C^\infty(M)$-bilinearity of $\cdot_A$. These two facts, and
nothing else, drive the whole section.

\begin{remark}\label{rmk:two-definitions}
	An earlier and slightly different notion of $F$-algebroid appears in
	\cite{CruzMorales2019}. There the base manifold is assumed to carry an
	$F$-manifold structure $(M,\bullet)$ and the anchor is required to be
	multiplicative,
	\[
		a_A(X\cdot_A Y)=a_A(X)\bullet a_A(Y),
	\]
	see \cite[Rem.~2.8]{CruzMorales2023}. We work with
	\Cref{def:Falgebroid}. All the structures produced below also satisfy the
	multiplicativity condition, as we point out where relevant.
\end{remark}

\begin{example}[{\cite[Ex.~2.9]{CruzMorales2023}}]\label{ex:tangent-F}
	Every $F$-manifold algebra is an $F$-algebroid over a point. If $(M,\bullet)$
	is an $F$-manifold, then $(TM,[-,-]_{\mathfrak X(M)},\bullet,\mathrm{Id})$ is an
	$F$-algebroid, the \emph{tangent $F$-algebroid} of $M$.
\end{example}

Here we introduce the transposed condition.

\begin{definition}\label{def:transposed-F}
	An $F$-algebroid $(A,[-,-]_A,\cdot_A,a_A)$ is \emph{transposed} if
	\begin{equation}\label{eq:transposed-F}
		2\,Z\cdot_A[X,Y]_A
		=
		[Z\cdot_A X,Y]_A+[X,Z\cdot_A Y]_A
		\qquad\text{for all } X,Y,Z\in\Gamma(A).
	\end{equation}
\end{definition}

Equivalently, $(\Gamma(A),\cdot_A,[-,-]_A)$ is a transposed Poisson algebra over
$\mathbb R$. The condition can also be written in terms of the tensor
\eqref{eq:Ptensor} alone, which is the form in which it compares directly with
the Hertling--Manin relation.

\begin{proposition}\label{prop:HM-characterization}
	Let $A$ be a commutative associative algebra endowed with a Lie bracket
	$[-,-]$, and let $P$ be as in \eqref{eq:Ptensor}. Then $(A,\cdot,[-,-])$ is a
	transposed Poisson algebra if and only if
	\[
		P_x(z,y)-P_y(z,x)=x[y,z]-y[x,z]
		\qquad\text{for all }x,y,z\in A.
	\]
\end{proposition}

\begin{proof}
	For all $x,y,z\in A$,
	\[
		\begin{aligned}
			P_x(z,y)-P_y(z,x)
			  & =[x,zy]-[x,z]y-z[x,y]-[y,zx]+[y,z]x+z[y,x] \\
			  & =[x,zy]+[zx,y]-y[x,z]+x[y,z]-2z[x,y],      
		\end{aligned}
	\]
	so that
	\[
		P_x(z,y)-P_y(z,x)-x[y,z]+y[x,z]=[x,zy]+[zx,y]-2z[x,y].
	\]
	The displayed identity is therefore equivalent to
	$[zx,y]+[x,zy]=2z[x,y]$, which is the transposed Leibniz rule.
\end{proof}

The Hertling--Manin relation does not have to be imposed separately.

\begin{proposition}\label{prop:HM-free}
	Let $A\to M$ be a vector bundle carrying a commutative associative algebroid
	structure $\cdot_A$ and a Lie algebroid structure $([-,-]_A,a_A)$ satisfying
	\eqref{eq:transposed-F}. Then $(A,[-,-]_A,\cdot_A,a_A)$ is a transposed
	$F$-algebroid.
\end{proposition}

\begin{proof}
	Only the Hertling--Manin relation for $(\Gamma(A),[-,-]_A,\cdot_A)$ has to be
	checked, and it holds because every transposed Poisson algebra is an
	$F$-manifold algebra, see \cite[\S2.4]{DamasBeites2024} and \cite{Bai2023}.
\end{proof}

Thus a transposed $F$-algebroid is exactly a Lie algebroid together with a
compatible commutative associative multiplication on its sections, the
compatibility being \eqref{eq:transposed-F}. This is the same economy as in the
algebraic case, where the transposed Leibniz rule already implies the
Hertling--Manin relation.

Substructures behave as expected.

\begin{definition}
	Let $(A,[-,-]_A,\cdot_A,a_A)$ be an $F$-algebroid. A vector subbundle
	$B\subseteq A$ is an \emph{$F$-subalgebroid} if $\Gamma(B)$ is closed under
	$\cdot_A$ and $[-,-]_A$, and $a_A(B)\subseteq TM$, so that the restricted data
	form an $F$-algebroid.
\end{definition}

\begin{proposition}
	An $F$-subalgebroid of a transposed $F$-algebroid is transposed.
\end{proposition}

\begin{proof}
	The product and the bracket of $B$ are the restrictions of those of $A$, and
	$\Gamma(B)$ is closed under both. Hence \eqref{eq:transposed-F} holds in
	$\Gamma(B)$ because it holds in $\Gamma(A)$.
\end{proof}

Before turning to the general structure, we record the examples in which the
anchor plays no role. They are all instances of one construction.

\begin{proposition}\label{prop:zero-anchor-classification}
	Let $(A,\cdot_A)$ be a commutative associative algebroid with identity $e$ and
	let $D\in\Gamma(\operatorname{End}(A))$ be a derivation of $\cdot_A$, that is
	$D(X\cdot_A Y)=D(X)\cdot_A Y+X\cdot_A D(Y)$. Then
	\[
		[X,Y]_D=X\cdot_A D(Y)-Y\cdot_A D(X),
		\qquad a_A=0,
	\]
	is a transposed $F$-algebroid structure on $A$, and every transposed
	$F$-algebroid with identity and vanishing anchor arises in this way, from
	$D=[e,-]_A$.
\end{proposition}

\begin{proof}
	Since $D$ is $C^\infty(M)$-linear, the bracket $[-,-]_D$ is
	$C^\infty(M)$-bilinear, so \eqref{eq:algebroid-leibniz} holds with $a_A=0$. The
	Jacobi identity and \eqref{eq:transposed-F} hold pointwise on the fibres, by the
	ungraded computation recalled around \eqref{eq:derivation-bracket}, and the
	Hertling--Manin relation follows from \Cref{prop:HM-free}.
	
	Conversely, let $(A,[-,-]_A,\cdot_A,0)$ be transposed with identity $e$, and put
	$D=[e,-]_A$. By \eqref{eq:algebroid-leibniz} with $a_A=0$ the operator $D$ is
	$C^\infty(M)$-linear, hence tensorial, so it is induced by a smooth bundle
	endomorphism of $A$. Since $\Gamma(A)$ is a unital commutative algebra over
	$\mathbb R$ with unit $e$, the unital case of the ungraded theory applies to it
	and gives that $D$ is a derivation of $\cdot_A$ and that
	\[
		[X,Y]_A=X\cdot_A D(Y)-Y\cdot_A D(X).
	\]
	Both statements are identities between sections, and no fibrewise argument is
	needed.
\end{proof}

\begin{example}
	Let $B$ be a finite-dimensional transposed Poisson algebra and let $A=M\times B$
	be the trivial bundle. With the fibrewise operations and $a_A=0$, $A$ is a
	transposed $F$-algebroid. Taking $B$ unital, this is
	\Cref{prop:zero-anchor-classification} with $D$ the derivation of $B$ given by
	the unital classification.
\end{example}

\begin{example}
	Any commutative associative algebroid with zero bracket and zero anchor is a
	transposed $F$-algebroid, namely the case $D=0$ above.
\end{example}

These examples are inert: the anchor vanishes, so nothing of the base manifold
enters. The rest of the section shows that this is almost the only possibility.

The following identity is the engine of the section. It uses only
\eqref{eq:algebroid-leibniz}, the $C^\infty(M)$-bilinearity of $\cdot_A$ and
\eqref{eq:transposed-F}. In particular it requires no identity section.

\begin{lemma}\label{lem:anchor-obstruction}
	Let $(A,[-,-]_A,\cdot_A,a_A)$ be a transposed $F$-algebroid. Then, for all
	$X,Y,Z\in\Gamma(A)$ and all $f\in C^\infty(M)$,
	\begin{equation}\label{eq:anchor-obstruction}
		a_A(X)(f)\,\bigl(Z\cdot_A Y\bigr)
		=
		a_A(Y)(f)\,\bigl(Z\cdot_A X\bigr).
	\end{equation}
\end{lemma}

\begin{proof}
	Write \eqref{eq:transposed-F} with $fZ$ in place of $Z$. Since $\cdot_A$ is
	$C^\infty(M)$-bilinear, the left-hand side is $2f\,Z\cdot_A[X,Y]_A$. On the
	right-hand side, $(fZ)\cdot_A X=f(Z\cdot_A X)$, and
	\eqref{eq:algebroid-leibniz} together with skew-symmetry gives
	\[
		[f(Z\cdot_A X),Y]_A=f[Z\cdot_A X,Y]_A-a_A(Y)(f)(Z\cdot_A X),
		\qquad
		[X,f(Z\cdot_A Y)]_A=f[X,Z\cdot_A Y]_A+a_A(X)(f)(Z\cdot_A Y).
	\]
	Subtracting $f$ times \eqref{eq:transposed-F} leaves
	$a_A(X)(f)(Z\cdot_A Y)-a_A(Y)(f)(Z\cdot_A X)=0$.
\end{proof}

\Cref{eq:anchor-obstruction} says that, as soon as the anchor is non-zero, the
multiplication is forced to be extremely degenerate.

\begin{proposition}\label{prop:rank-one-product}
	Let $(A,[-,-]_A,\cdot_A,a_A)$ be a transposed $F$-algebroid and let $p\in M$ be
	a point at which $a_A$ does not vanish. Then there are a neighbourhood $U$ of
	$p$, a nowhere vanishing $\varphi\in\Gamma(A^*|_U)$, a section
	$u\in\Gamma(A|_U)$ and a vector field $W$ on $U$ such that
	\[
		Z\cdot_A Y=\varphi(Z)\varphi(Y)\,u
		\qquad\text{on }U .
	\]
	If moreover $u_p\neq0$, then in addition $a_A=\varphi\otimes W$ on a possibly
	smaller neighbourhood, with $W_p\neq0$.
\end{proposition}

\begin{proof}
	Choose $X_0\in\Gamma(A)$ and $f_0\in C^\infty(M)$ with
	$a_A(X_0)(f_0)(p)\neq0$, and let $U$ be a neighbourhood of $p$ on which
	$a_A(X_0)(f_0)$ does not vanish. Since $a_A$ is a bundle map,
	\[
		\varphi(Y):=\frac{a_A(Y)(f_0)}{a_A(X_0)(f_0)}
	\]
	is $C^\infty(U)$-linear in $Y$, hence defines $\varphi\in\Gamma(A^*|_U)$, and
	$\varphi(X_0)=1$.
	
	Taking $X=X_0$ and $f=f_0$ in \eqref{eq:anchor-obstruction} gives
	$Z\cdot_A Y=\varphi(Y)(Z\cdot_A X_0)$. By commutativity,
	$\varphi(Y)(Z\cdot_A X_0)=\varphi(Z)(Y\cdot_A X_0)$, and putting $Z=X_0$ yields
	$Y\cdot_A X_0=\varphi(Y)u$ with $u:=X_0\cdot_A X_0$. Substituting back,
	$Z\cdot_A Y=\varphi(Z)\varphi(Y)u$. In particular $\varphi$ is determined by
	$\cdot_A$ once $X_0$ is chosen, and does not depend on $f_0$.
	
	Assume now $u_p\neq0$ and shrink $U$ so that $u$ does not vanish. Feeding the
	last formula into \eqref{eq:anchor-obstruction} gives
	\[
		a_A(X)(f)\,\varphi(Y)\varphi(Z)\,u
		=a_A(Y)(f)\,\varphi(X)\varphi(Z)\,u ,
	\]
	and taking $Z=X_0$, $Y=X_0$ leaves $a_A(X)(f)=\varphi(X)\,W(f)$ with
	$W:=a_A(X_0)$. Finally $W_p\neq0$, since otherwise $a_A$ would vanish at $p$.
\end{proof}

\begin{remark}\label{rmk:rank-one-normalisation}
The triple $(\varphi,u,W)$ is not canonical. As the proof shows, $\varphi$ does
not depend on $f_0$, so the only freedom is the choice of $X_0$. Replacing
$X_0$ by a section $X_0'$ with $\varphi(X_0')=\lambda$ nowhere vanishing gives
\[
	\varphi'=\lambda^{-1}\varphi,
	\qquad
	u'=X_0'\cdot_A X_0'=\lambda^{2}u,
	\qquad
	W'=a_A(X_0')=\lambda W,
\]
the first because $a_A(X_0')=\lambda W$ and the second by the product formula.
Both $a_A=\varphi\otimes W$ and the product $Z\cdot_A Y=\varphi(Z)\varphi(Y)u$
are unchanged, as they must be. The normalisation is fixed once and for all in
\Cref{thm:local-model}, where $X_0$ is rescaled so that it becomes an identity.
\end{remark}

\begin{theorem}\label{thm:rank-one}
	Let $(A,[-,-]_A,\cdot_A,a_A)$ be a transposed $F$-algebroid and let $p\in M$. If
	$a_A$ does not vanish at $p$ and the multiplication $\cdot_A$ does not vanish
	identically on the fibre $A_p$, then
	\[
		\operatorname{rank}_p A=1 .
	\]
\end{theorem}

\begin{proof}
	Keep the notation of \Cref{prop:rank-one-product}. By that proposition the
	multiplication on $A_p$ is $\varphi(Z)\varphi(Y)u_p$, so the hypothesis on
	$\cdot_A$ says exactly that $u_p\neq0$. Hence $a_A=\varphi\otimes W$ near $p$,
	with $W_p\neq0$.
	
	Insert $Z\cdot_A Y=\varphi(Z)\varphi(Y)u$ into \eqref{eq:transposed-F}. The left
	hand side is $2\varphi(Z)\varphi([X,Y]_A)u$, while
    \begin{align*}
        [Z\cdot_A X,Y]_A+[X,Z\cdot_A Y]_A
		&=\varphi(Z)\varphi(X)[u,Y]_A-\varphi(Z)\varphi(Y)[u,X]_A\\
        \quad &+\Bigl(a_A(X)\bigl(\varphi(Z)\varphi(Y)\bigr)
		-a_A(Y)\bigl(\varphi(Z)\varphi(X)\bigr)\Bigr)u ,
    \end{align*}
	by \eqref{eq:algebroid-leibniz}. Using $a_A=\varphi\otimes W$, the last bracket
	equals $\varphi(Z)\bigl(\varphi(X)W(\varphi(Y))-\varphi(Y)W(\varphi(X))\bigr)$.
	Taking $Z=X_0$, so that $\varphi(Z)=1$, we obtain
	\begin{equation}\label{eq:dagger}
		\varphi(X)[u,Y]_A-\varphi(Y)[u,X]_A
		+\bigl(\varphi(X)W(\varphi(Y))-\varphi(Y)W(\varphi(X))\bigr)u
		=2\varphi([X,Y]_A)\,u .
	\end{equation}
	Setting $X=X_0$ in \eqref{eq:dagger} and using $W(1)=0$ gives
	\begin{equation}\label{eq:ddagger}
		[u,Y]_A=\varphi(Y)[u,X_0]_A+c(Y)\,u,
		\qquad
		c(Y):=2\varphi([X_0,Y]_A)-W(\varphi(Y)).
	\end{equation}
	We now compare the two ways of computing $[u,fY]_A$. On one hand
	\eqref{eq:algebroid-leibniz} gives
	$[u,fY]_A=f[u,Y]_A+\varphi(u)W(f)Y$. On the other hand \eqref{eq:ddagger}
	applied to $fY$ gives $[u,fY]_A=f\varphi(Y)[u,X_0]_A+c(fY)u$. Moreover
	\[
		\varphi([X_0,fY]_A)=f\varphi([X_0,Y]_A)+W(f)\varphi(Y),
		\qquad
		W(\varphi(fY))=fW(\varphi(Y))+W(f)\varphi(Y),
	\]
	so that $c(fY)=fc(Y)+W(f)\varphi(Y)$. Comparing the two expressions for
	$[u,fY]_A$ and using \eqref{eq:ddagger} once more, we are left with
	\[
		W(f)\,\varphi(Y)\,u=\varphi(u)\,W(f)\,Y
		\qquad\text{for all } f\in C^\infty(M),\ Y\in\Gamma(A).
	\]
	Choose $f$ with $W(f)(p)\neq0$, which is possible because $W_p\neq0$. Then
	\begin{equation}\label{eq:section-sign}
		\varphi(Y)\,u_p=\varphi(u)\,Y_p
		\qquad\text{for all } Y\in\Gamma(A).
	\end{equation}
	If $\varphi(u)(p)=0$, then $\varphi(Y)u_p=0$ for every $Y$, and taking $Y=X_0$
	gives $u_p=0$, contrary to the hypothesis. Hence $\varphi(u)(p)\neq0$ and
	\eqref{eq:section-sign} expresses every $Y_p$ as a multiple of $u_p$. Therefore
	$A_p$ is spanned by $u_p$.
\end{proof}

Rank one is not only forced, but it also determines the structure completely.

\begin{theorem}\label{thm:local-model}
	Let $(A,[-,-]_A,\cdot_A,a_A)$ be a transposed $F$-algebroid and let $p\in M$ be
	a point at which $a_A$ and $\cdot_A$ do not vanish. Then on a neighbourhood $U$
	of $p$ there are a nowhere vanishing $e\in\Gamma(A|_U)$ and a nowhere vanishing
	vector field $V$ on $U$ such that $A|_U$ is the trivial line bundle generated by
	$e$ and
	\[
		(fe)\cdot_A(ge)=fg\,e,
		\qquad
		a_A(fe)=fV,
		\qquad
		[fe,ge]_A=\bigl(fV(g)-gV(f)\bigr)e .
	\]
	In particular $e$ is a local identity and $V=a_A(e)$. Conversely, for every
	manifold $U$ and every vector field $V$ on $U$, these formulas define a
	transposed $F$-algebroid structure on the trivial line bundle over $U$.
\end{theorem}

\begin{proof}
	By \Cref{thm:rank-one} the bundle has rank one near $p$, and by
	\Cref{prop:rank-one-product} we have $Z\cdot_A Y=\varphi(Z)\varphi(Y)u$ with
	$\varphi(X_0)=1$ and $u_p\neq0$. Identity \eqref{eq:section-sign} reads
	$\varphi(Y)u=\varphi(u)Y$. Put $c:=\varphi(u)$. Taking $Y=X_0$ gives $u=cX_0$,
	and $c$ is nowhere vanishing near $p$, since $\varphi(u)(p)\neq0$ was
	established in the proof of \Cref{thm:rank-one}. Combining the two identities,
	\[
		Z\cdot_A Y=\varphi(Z)\varphi(Y)u=c\,\varphi(Z)\,Y ,
	\]
	so that $X_0$ is an identity only if $c=1$. We therefore normalise differently
	and set
	\[
		e:=c^{-1}X_0=c^{-2}u,
		\qquad\text{so that}\qquad
		\varphi(e)=c^{-1}.
	\]
	Then $e\cdot_A Y=c\,\varphi(e)\,Y=Y$, so $e$ is an identity on a neighbourhood
	of $p$, it vanishes nowhere, and $(fe)\cdot_A(ge)=fg\,e$.
	
	By \Cref{prop:rank-one-product} the anchor is $a_A=\varphi\otimes W$ with
	$W=a_A(X_0)$, whence $a_A(fe)=f\varphi(e)W=fV$ for
	\[
		V:=a_A(e)=c^{-1}W ,
	\]
	which is nowhere zero near $p$ because $W_p\neq0$. Finally $D:=[e,-]_A$
	satisfies $D(ge)=g[e,e]_A+a_A(e)(g)e=V(g)e$, and the unital case of the ungraded
	theory gives $[X,Y]_A=X\cdot_A D(Y)-Y\cdot_A D(X)$. Using $e\cdot_A e=e$,
	\[
		[fe,ge]_A=(fe)\cdot_A V(g)e-(ge)\cdot_A V(f)e
		=\bigl(fV(g)-gV(f)\bigr)e .
	\]
	
	For the converse, the displayed bracket is the transposed Poisson bracket
	attached to the derivation $V$ of $C^\infty(U)$, so it is a Lie bracket
	satisfying \eqref{eq:transposed-F}. The Leibniz rule
	\eqref{eq:algebroid-leibniz} holds with anchor $fe\mapsto fV$, and
	$a_A([fe,ge]_A)=\bigl(fV(g)-gV(f)\bigr)V=[fV,gV]$, so $A$ is a Lie algebroid.
	The Hertling--Manin relation follows from \Cref{prop:HM-free}.
\end{proof}

\begin{remark}
	The local model also satisfies the multiplicativity condition
	$a_A(X\cdot_A Y)=a_A(X)\bullet a_A(Y)$ of \Cref{rmk:two-definitions}, for the
	product $\bullet$ on the line distribution generated by $V$ determined by
	$V\bullet V=V$. So the classification is insensitive to which of the two
	definitions of $F$-algebroid is adopted.
\end{remark}

Putting the two halves together, we obtain a complete description in the unital
case.

\begin{proposition}\label{cor:unital-dichotomy}
	Let $(A,[-,-]_A,\cdot_A,e,a_A)$ be a transposed $F$-algebroid with identity over
	a connected base $M$, with $A\neq0$. Then exactly one of the following holds.
	\begin{enumerate}
		\item $\operatorname{rank}A\geq2$. Then $a_A=0$, and the structure is the one of
		      \Cref{prop:zero-anchor-classification}, determined by the derivation
		      $D=[e,-]_A$ of $(A,\cdot_A)$.
		\item $\operatorname{rank}A=1$. Then $e$ trivialises $A$, and the structure is
		      determined on all of $M$ by the single vector field $V=a_A(e)$ through
		      \[
		      	(fe)\cdot_A(ge)=fg\,e,
		      	\qquad
		      	a_A(fe)=fV,
		      	\qquad
		      	[fe,ge]_A=\bigl(fV(g)-gV(f)\bigr)e .
		      \]
		      Conversely, every vector field on $M$ arises this way, and $V$ is allowed to
		      vanish, on a closed subset or identically.
	\end{enumerate}
	In particular, a transposed $F$-algebroid with identity and rank at least two has
	vanishing anchor, and one of rank one is globally the Jacobi model of a single
	vector field.
\end{proposition}

\begin{proof}
	The rank of $A$ is locally constant, hence constant on $M$, and it is at least
	one because $A\neq0$. So exactly one of the two cases occurs.
	
	Assume $\operatorname{rank}A\geq2$. The identity gives $e\cdot_Ae=e\neq0$, so
	$\cdot_A$ vanishes at no point. If $a_A$ were non-zero at some $p$, then
	\Cref{thm:rank-one} would give $\operatorname{rank}_pA=1$, a contradiction.
	Hence $a_A=0$, and \Cref{prop:zero-anchor-classification} applies.
	
	Assume $\operatorname{rank}A=1$. If $e_p=0$ for some $p$, then
	$Y_p=e_p\cdot_AY_p=0$ for every $Y\in\Gamma(A)$, so $A_p=0$, against
	$\operatorname{rank}A=1$. Thus $e$ is a nowhere vanishing section of a line
	bundle, it trivialises $A$, and $\Gamma(A)=C^\infty(M)\,e$ with
	$(fe)\cdot_A(ge)=fg\,e$. Since $a_A$ is a bundle map, $a_A(fe)=f\,a_A(e)=fV$.
	
	The triple $(\Gamma(A),\cdot_A,[-,-]_A)$ is a unital transposed Poisson algebra,
	so the unital case of the ungraded theory gives
	$[X,Y]_A=X\cdot_AD(Y)-Y\cdot_AD(X)$ with $D=[e,-]_A$. By
	\eqref{eq:algebroid-leibniz} and $[e,e]_A=0$ we have
	$D(ge)=g[e,e]_A+a_A(e)(g)\,e=V(g)\,e$, whence
	$[fe,ge]_A=\bigl(fV(g)-gV(f)\bigr)e$. None of these steps requires $V$ to be
	non-zero. The converse is the second half of \Cref{thm:local-model}, applied
	with $U=M$.
\end{proof}

The contrast with \Cref{thm:degree-zero-bundle} is worth recording. On
$\Gamma(\Lambda^\bullet E)$ the transposed brackets of degree zero are the
gauge derivations of $E$, that is the first order operators with an arbitrary
symbol. Here the anchor is exactly the symbol of $D=[e,-]_A$, and
\Cref{cor:unital-dichotomy} says that it is forced to vanish unless the rank is
one. Transposedness tolerates a symbol only in rank one.

The tangent $F$-algebroid of \Cref{ex:tangent-F} has anchor the identity, so it
is never inert. \Cref{thm:rank-one} then determines it completely.

\begin{corollary}\label{cor:F-manifold}
	Let $(M,\bullet,e)$ be an $F$-manifold with identity, over a connected base. Its
	tangent $F$-algebroid $(TM,[-,-]_{\mathfrak X(M)},\bullet,e,\mathrm{Id})$ is
	transposed if and only if $\dim M=1$.
\end{corollary}

\begin{proof}
	The anchor is the identity, hence nowhere vanishing, and $e\bullet e=e\neq0$, so
	the multiplication is nowhere vanishing. If the tangent $F$-algebroid is
	transposed, \Cref{thm:rank-one} gives $\dim T_pM=1$ for every $p$.
	
	Conversely, let $\dim M=1$. Then $e$ is a nowhere vanishing vector field, every
	vector field is $fe$, and $e\bullet e=e$ forces $(fe)\bullet(ge)=fge$. The Lie
	bracket of vector fields is $[fe,ge]=(fe(g)-ge(f))e$, which is the local model
	of \Cref{thm:local-model} with $V=e$. Hence the tangent $F$-algebroid is
	transposed.
\end{proof}

This supersedes the semisimple computation, in which the same conclusion is
reached under the additional hypothesis that $\bullet$ be semisimple on a dense
open subset. No semisimplicity is needed, and no canonical coordinates are used.

\begin{remark}
	The one-dimensional case is not vacuous. On $M=\mathbb R$ with coordinate $u$
	and $e=\partial_u$, the tangent $F$-algebroid is transposed, its bracket being
	the Wronskian $[f\partial_u,g\partial_u]=(fg'-f'g)\partial_u$. This is the
	smallest example of a transposed $F$-algebroid with invertible anchor, and by
	\Cref{cor:F-manifold} it is essentially the only one of tangent type.
\end{remark}

Rank one is not an accident of the computation. It points at the geometry in
which the transposed identity is at home.

Let $L\to M$ be a line bundle. Recall that a \emph{Jacobi bundle} structure on
$L$ is a Lie bracket on $\Gamma(L)$ which is a first order differential operator
in each argument, in the sense of Kirillov's local Lie algebras
\cite{Kirillov1976} and of Marle \cite{Marle1991}. For the trivial line bundle
such a bracket is
\begin{equation}\label{eq:jacobi}
	\{f,g\}_{(\Lambda,V)}=\Lambda(df,dg)+fV(g)-gV(f),
\end{equation}
for a bivector $\Lambda$ and a vector field $V$ subject to the usual
integrability conditions.

\begin{proposition}\label{prop:jacobi}
	Let $(A,[-,-]_A,\cdot_A,e,a_A)$ be a transposed $F$-algebroid with identity and
	$\operatorname{rank}A=1$, over a connected base $M$. Then $e$ is a nowhere vanishing global
	section, so $A$ is the trivial line bundle, and under the induced identification
	$\Gamma(A)\cong C^\infty(M)$ the bracket is the Jacobi bracket
	\eqref{eq:jacobi} with $\Lambda=0$ and $V=a_A(e)$. Conversely every vector field
	on $M$ arises in this way.
\end{proposition}

\begin{proof}
	By \Cref{cor:unital-dichotomy} the section $e$ vanishes nowhere and
	$[fe,ge]_A=(fV(g)-gV(f))e$ with $V=a_A(e)$. Under $fe\mapsto f$ this is
	\eqref{eq:jacobi} with $\Lambda=0$. The converse is the second half of
	\Cref{thm:local-model}.
\end{proof}

The same phenomenon was met in \Cref{prop:transposed-manifold}, where the
transposed Poisson structures on a manifold turned out to be the Jacobi brackets
with vanishing bivector part. \Cref{prop:jacobi} says that the algebroid level
does not enlarge the picture: passing from functions to sections of a vector
bundle, and from a Lie bracket to a Lie algebroid bracket, produces nothing new
as long as the anchor survives.

\begin{remark}
	The natural Lie algebroid of Jacobi geometry is the gauge algebroid
	$\mathfrak D(L)$ of first order operators on a line bundle
	\cite{Kirillov1976,Marle1991}, and Jacobi structures on $L$ correspond to
	suitable bivectors on $\mathfrak D(L)$. \Cref{prop:jacobi} places the transposed
	condition inside this picture as the locus $\Lambda=0$. This is consistent with
	\Cref{thm:degree-zero-bundle}, where the classifying object in degree zero was
	again a gauge algebroid.
\end{remark}

We close by recording how the transposed condition interacts with the standard
constructions of $F$-algebroid theory. Throughout this subsection
$(A,[-,-]_A,\cdot_A,e,a_A)$ is an $F$-algebroid with an identity $e$. We first
show that Dubrovin-type twists and Nijenhuis deformations preserve
transposedness, and then that the derivation construction has a pre-Lie origin.
After \Cref{thm:local-model} none of this is surprising, since the bracket of a
transposed $F$-algebroid with identity is a Wronskian.

Recall that the \emph{Hertling--Manin tensor} of $A$ is
\[
        P_X(Y,Z)=[X,Y\cdot_AZ]_A-[X,Y]_A\cdot_AZ-Y\cdot_A[X,Z]_A .
\]
A section $E\in\Gamma(A)$ is a \emph{pseudoeventual identity}
\cite[Def.~4.1]{CruzMorales2023} if
\begin{equation}\label{eq:pseudoeventual}
        P_E(X,Y)=[e,E]_A\cdot_AX\cdot_AY,
        \qquad X,Y\in\Gamma(A).
\end{equation}
The tensor $P_E$ measures the failure of $[E,-]_A$ to be a derivation of
$\cdot_A$. The condition expressed by \Cref{eq:pseudoeventual} requires this failure to be the
multiplication operator attached to the single section $[e,E]_A$. The identity
$e$ is itself a pseudoeventual identity, since $P_e=0$ and $[e,e]_A=0$. A
pseudoeventual identity is called an \emph{eventual identity} when it is
invertible, that is when $E^{-1}\cdot_AE=e$ for some $E^{-1}\in\Gamma(A)$.

By \cite[Thm.~4.3]{CruzMorales2023}, a section $E$ is a pseudoeventual identity
if and only if
\[
        X\cdot_EY=X\cdot_AY\cdot_AE
\]
turns $(A,[-,-]_A,\cdot_E,a_A)$ into an $F$-algebroid. This new $F$-algebroid
carries an identity precisely when $E$ is invertible, in which case the identity
is $E^{-1}$ and the construction is Dubrovin's dual
\cite[Thm.~4.4]{CruzMorales2023}.

In the transposed case the notion is not restrictive at all.

\begin{proposition}\label{prop:all-pseudoeventual}
Let $(A,[-,-]_A,\cdot_A,e,a_A)$ be a transposed $F$-algebroid with an identity.
Then every section of $A$ is a pseudoeventual identity.
\end{proposition}

\begin{proof}
By \Cref{thm:local-model} the bracket is the Wronskian
$[X,Y]_A=X\cdot_AD(Y)-Y\cdot_AD(X)$ attached to $D=[e,-]_A$, and $D$ is a
derivation of $(\Gamma(A),\cdot_A)$. For $E,X,Y\in\Gamma(A)$,
\[
\begin{aligned}
\relax[E,X\cdot_AY]_A
&=E\cdot_AD(X)\cdot_AY+E\cdot_AX\cdot_AD(Y)-X\cdot_AY\cdot_AD(E), \\
[E,X]_A\cdot_AY&=E\cdot_AD(X)\cdot_AY-X\cdot_AY\cdot_AD(E), \\
X\cdot_A[E,Y]_A&=E\cdot_AX\cdot_AD(Y)-X\cdot_AY\cdot_AD(E).
\end{aligned}
\]
Subtracting the last two lines from the first gives
$P_E(X,Y)=D(E)\cdot_AX\cdot_AY$, and $D(E)=[e,E]_A$ by definition of $D$.
\end{proof}

\begin{proposition}
Let $(A,[-,-]_A,\cdot_A,e,a_A)$ be a transposed $F$-algebroid and let
$E\in\Gamma(A)$. Then $(A,[-,-]_A,\cdot_E,a_A)$ is transposed. In particular,
whenever $E$ is invertible, the Dubrovin dual is again transposed.
\end{proposition}

\begin{proof}
By \Cref{prop:all-pseudoeventual} the section $E$ is a pseudoeventual identity,
so $(A,[-,-]_A,\cdot_E,a_A)$ is an $F$-algebroid by
\cite[Thm.~4.3]{CruzMorales2023}. For $X,Y,Z\in\Gamma(A)$ we have
$2Z\cdot_E[X,Y]_A=2(Z\cdot_AE)\cdot_A[X,Y]_A$. Applying
\eqref{eq:transposed-F} with $Z\cdot_AE$ in place of $Z$, and using
commutativity and associativity of $\cdot_A$ in the form
$(Z\cdot_AE)\cdot_AX=Z\cdot_EX$, this equals
$[Z\cdot_EX,Y]_A+[X,Z\cdot_EY]_A$.
\end{proof}

Recall from \cite[Prop.~4.10]{CruzMorales2023} that $N=E\cdot_A-$ is a Nijenhuis
operator on $A$, with deformed bracket
$[X,Y]_E=[E\cdot_AX,Y]_A+[X,E\cdot_AY]_A-E\cdot_A[X,Y]_A$ and deformed anchor
$a_E(X)=a_A(E\cdot_AX)$.

\begin{proposition}
Let $(A,[-,-]_A,\cdot_A,e,a_A)$ be a transposed $F$-algebroid and let
$E\in\Gamma(A)$. Then $[X,Y]_E=E\cdot_A[X,Y]_A$, and the deformed
$F$-algebroid $(A,[-,-]_E,\cdot_E,a_E)$ is transposed.
\end{proposition}

\begin{proof}
	Taking $Z=E$ in \eqref{eq:transposed-F} gives
	$[E\cdot_AX,Y]_A+[X,E\cdot_AY]_A=2E\cdot_A[X,Y]_A$, whence
	$[X,Y]_E=E\cdot_A[X,Y]_A$. Therefore
	\[
		2Z\cdot_E[X,Y]_E=2Z\cdot_AE^2\cdot_A[X,Y]_A,
	\]
	while
	\[
		[Z\cdot_EX,Y]_E+[X,Z\cdot_EY]_E
		=E\cdot_A\bigl([Z\cdot_AE\cdot_AX,Y]_A+[X,Z\cdot_AE\cdot_AY]_A\bigr)
		=2Z\cdot_AE^2\cdot_A[X,Y]_A,
	\]
	by \eqref{eq:transposed-F} applied with $Z\cdot_AE$ in place of $Z$.
\end{proof}

\begin{remark}
The identity $[X,Y]_E=E\cdot_A[X,Y]_A$ says that the Nijenhuis deformation
attached to $E$ is the rescaling of the bracket by $E$, performed
simultaneously on the bracket and on the product. Transposedness is stable
under rescaling of the bracket alone, as already observed in
\Cref{eq:rescaling}, and the proposition shows that the two rescalings do not
interfere.
\end{remark}

Finally, the derivation construction of
\Cref{prop:zero-anchor-classification} has a pre-Lie origin, in the sense of
\cite[\S3]{CruzMorales2023}.

\begin{proposition}
	Let $(A,\cdot_A)$ be a commutative associative algebroid and let
	$D\in\Gamma(\operatorname{End}(A))$ be a derivation of $\cdot_A$. Then
	$X*Y=X\cdot_AD(Y)$ is a pre-Lie-com algebroid structure with zero anchor, and
	its subadjacent $F$-algebroid is the transposed $F$-algebroid of
	\Cref{prop:zero-anchor-classification}.
\end{proposition}

\begin{proof}
	The associator of $*$ is
	$X*(Y*Z)-(X*Y)*Z=X\cdot_AY\cdot_AD^2(Z)$, which is symmetric in $X$ and $Y$, so
	$*$ is pre-Lie, and it is $C^\infty(M)$-bilinear, so the anchor vanishes.
	Moreover
	$X*(Y\cdot_AZ)=(X*Y)\cdot_AZ+Y\cdot_A(X*Z)$, so $(A,*,\cdot_A,0)$ is
	pre-Lie-com. Its subadjacent bracket is
	$X*Y-Y*X=X\cdot_AD(Y)-Y\cdot_AD(X)$.
\end{proof}

\section{Open problems}\label{sec:problems}

We would like to conclude this work with a list of open problems related to these topics.

\begin{problem}\label{pb:nonunital}
Every classification in this paper is obtained by evaluating
\eqref{eq:intro-transposed} at a unit, and \Cref{prop:exterior-obstruction} is
the only obstruction whose proof does not. Over a base without unit the question
is first of all one of definitions: for a non-unital commutative algebra $A$ one
has to decide which category of $A$-modules to work in, what $\Lambda^0_AL$
should be, and whether the Schouten--Rinehart bracket still extends to
$\Lambda_AL$. Once these are fixed, is the degree one obstruction
\[
	[Z,Y]\wedge X+[X,Z]\wedge Y=0,
	\qquad X,Y,Z\in L,
\]
also sufficient, and does the exterior annihilator $\operatorname{Ann}_\wedge(L)$
classify the transposed brackets which survive? A unit is needed even for the
injectivity of the correspondence of \Cref{thm:unital-odd}: on the non-unital
algebra $A=\langle\theta\rangle$ with $\theta^2=0$ and $|\theta|=-1$, the
element $\theta$ is non-zero while the associated bracket is not.
\end{problem}

\begin{problem}\label{pb:halfder}
Outside the unital case the classification has to proceed through the graded
$\frac12$-derivations of the graded Lie part, as in
\Cref{prop:no-half-derivations}. What are the graded $\frac12$-derivations of the
Lie superalgebras of Cartan type, and of the Virasoro-type algebras of
\cite{Kaygorodov2025VirasoroType}?
\end{problem}

\begin{problem}\label{pb:oddweak}
\Cref{thm:unital-odd} says that in odd degree the transposed Leibniz rule is too
strong to coexist with a non-negative grading. Is there a weaker compatibility
which retains the Gerstenhaber degree and is not vacuous? The Novikov and
anti-pre-Lie superstructures of \cite{Chen2024}, together with the bialgebra
theory of \cite{Liu2024}, are the natural candidates.
\end{problem}

\begin{problem}\label{pb:compatible}
Pairs of compatible structures are a recurrent theme in Poisson geometry,
starting from Poisson--Nijenhuis manifolds \cite{Kosmann-Schwarzbach1990}. On a
single graded commutative algebra, is there a meaningful notion of compatibility
between a Gerstenhaber bracket and a transposed bracket of degree zero, and does
it generate a hierarchy?
\end{problem}

\section*{Acknowledgements}

The author is deeply grateful to the Instituto de Matem\'atica Pura e Aplicada
(IMPA), Rio de Janeiro, for its hospitality and support throughout his visit,
during which this work was begun. The visit was made possible by the Istituto
Nazionale di Alta Matematica ``Francesco Severi'' (INdAM), which funded the
research period jointly with IMPA through the INdAM--IMPA Call 2026 for study
and scientific visits in Brazil. Particular thanks go to Prof.~Henrique
Bursztyn (IMPA) for his generous welcome and for the discussions from which
several of the questions addressed here arose. The author is a member of the
``Gruppo Nazionale per le Strutture Algebriche, Geometriche e le loro
Applicazioni'' (GNSAGA--INdAM) and he is also supported by the University of Palermo (FFR2026, UNIPA BsD).

\bibliographystyle{amsalpha}
\bibliography{references}

\end{document}